\documentclass{article}
\usepackage[letterpaper, textwidth=5.0in]{geometry}

\usepackage{amssymb,amsbsy,amsmath,amsfonts,amssymb,amscd,amsthm}
\usepackage{graphicx}
\graphicspath{
  {images/}{../images/}
}
\usepackage{xcolor}
\usepackage{cancel}
\usepackage{hyperref}
\usepackage{ulem}
\usepackage{subcaption}
\usepackage{float}
\usepackage{comment}

\newtheorem{dummy}{dummy}[section]
\newtheorem{Theorem}[dummy]{Theorem}
\newtheorem{proposition}[dummy]{Proposition}
\newtheorem*{Proposition*}{Proposition}
\newtheorem{lemma}[dummy]{Lemma}

\newtheorem{defi}[dummy]{Definition}
\newtheorem{Remark}[dummy]{Remark}

\title{
Galerkin method for asymmetrically-weighted Hermite approximations applied to the Vlasov-Poisson system
}
\author{
Ruiyang DAI%
\thanks{\url{ruiyang.dai@inria.fr}}
}
\date{}

\begin{document}
\maketitle
\begin{abstract}
We investigate a numerical method for the Vlasov-Poisson (VP) system utilizing asymmetrically-weighted (AW)
Hermite bases in velocity space, which is an hyperbolic system. 
In particular, we concentrate on spectral methods in velocity.
For the Hermite spectral form of the VP system,
we analyze the resaon that the form with AW Hermite bases can be instable.
To obtain $L^2$ stability properties, 
we consider a Galerkin method intead of the classical Petrov-Galerkin method,
which naturally ensures stability with respect to the $L^2$ norm.
We also present an equivalent form of the method that  
maintains a computational cost modest compared to that of the Petrov-Galerkin method.
Finally, we present numerical simulations based on the proposed Hermite spectral method, 
showcasing its stability.
\end{abstract}

\tableofcontents
\section{Introduction}
The use of plasmas in everyday life has become increasingly common. 
Examples include neon tubes and plasma screens, and plasmas are also widely used in various industrial applications, 
such as the fusion plasma \cite{miyamoto2005plasma}.
The one of the most significant programs is the ITER program 
(originally the International Thermonuclear Experimental Reactor), 
an international collaboration designed to demonstrate the feasibility of 
generating electricity through controlled nuclear fusion. 
A tokamak is a large toroidal chamber in which nuclear fusion reactions occur 
within a plasma confined by powerful magnetic fields. 
ITER will use a deuterium-tritium fuel mixture, 
selected because the most accessible fusion reaction involves 
the fusion of deuterium and tritium nuclei-two isotopes of hydrogen-producing 
a helium nucleus and a highly energetic neutron. 
The energy carried by this neutron can then be converted into heat and ultimately used to generate electricity.
Therefore, it is crucial to employ models that account for 
both the intrinsic collective dynamics of the plasma 
and the external forces that strongly shape the system's evolution.
Kinetic models often are used to describe the evolution of a plasma.
One of the models is currently applied in plasma physics simulations is the Vlasov-Maxwell system,
which describe microscopic plasma dynamics through the phase-space distribution function 
defined in seven dimensions: three spatial coordinates, three velocity coordinates, and time.

The VP system, which is a simplification of the Vlasov–Maxwell system, 
is one of the simplest model in plasma physics simulations.
Due to their high dimensionality, and plus nonlinearities, 
the VP system is challenging to solve numerically. 
Consequently, the development of accurate numerical methods for 
solving this system has been an active area of research since the 1960s.
Particle-in-cell (PIC) methods have been a popular and effective method for plasma physics \cite{birdsall2018plasma}.
PIC methods, approximating the kinetic simulations by a finite number of macro particles, 
effectively reduce the dimension from six to three. 
However, the main drawback of PIC methods is their inherent numerical error
associated with particle noise \cite{filbet2003numerical}, 
which decreases slowly when the number of particles increases. 
More specifically, the noise in PIC methods decreases in proportion to 
the inverse square root of the number of particles per cell.

To overcome this limitation, Eulerian solvers---that is, methods that discretize the Vlasov equation on a six-dimensional phase-space grid (three spatial and three velocity dimensions)---can be employed.
Their development has been extensively investigated in the literature, 
and comprehensive reviews of the various approaches, 
together with their respective advantages and limitations, 
can be found in \cite{filbet2003comparison,duclous2009high,sonnendrucker2004vlasov}.
These approaches include, among others, 
finite-volume methods \cite{filbet2001convergence}, 
Fourier–Fourier transform schemes \cite{klimas1994splitting}, 
and semi-Lagrangian schemes \cite{sonnendrucker1999semi}.

Using orthogonal polynomials in the velocity variable of Eulerian solvers results in spectral methods. 
The idea of representing the distribution function with a finite set of orthogonal polynomials 
using Galerkin or Petrov-Galerkin methods, 
rather than discretizing it directly in velocity space, dates back to the 1960s \cite{armstrong1966numerical,joyce1971numerical}.
Many Galerkin or Petrov-Galerkin methods for the VP system have focused on using Hermite polynomials in 
velocity space.
Moreover, in many plasma physics problems, 
the solutions exhibit exponential decay as $|v|\to\infty$.
In such cases, it is reasonable to employ Hermite functions, 
obtained by combining Hermite polynomials with a Gaussian function, 
as basis functions.
Motivated by this choice of basis, 
it is therefore natural to consider weight functions \(\omega(v)\) 
that yield orthogonal systems.
Based on the choice of the Gaussian function,
there are two main approaches of Hermite discretization in the velocity variable, 
both of which lead the VP system to an hyperbolic system: 
the  symmetrically-weighted (SW) 
and asymmetrically-weighted (AW) Hermite functions.
In \cite{holloway1996spectral}, Holloway formalized these two approaches. 
The first one is based on the SW Hermite functions as the basis in velocity 
and as test functions in the Galerkin method, 
corresponding to the choice of weight function $\omega(v)=1$.
It shows numerically that this SW method cannot simultaneously conserve mass, momentum and total energy. 
However, it conserves the $L^2$ norm of the distribution function, which ensures the stability of the method. 
The second approach utilizes AW Hermite functions as trial functions 
and a distinct set of test functions orthogonal to the AW Hermite basis 
in the Petrov–Galerkin method.
In this case, the test functions are constructed 
from AW Hermite functions together 
with a weight function depending on $v$. 
This approach yields the simultaneous conservation of 
mass, momentum, and total energy.
However, it also shows numerically that the method based on AW Hermite functions 
does not conserve the $L^2$ norm of the distribution function 
and is then not numerically stable. 
In \cite{kormann2021generalized}, K. Kormann and A. Yurova, 
using the idea of telescoping sums to show conservation 
for SW and AW Hermite functions, 
reaches conclusions consistent with those in \cite{holloway1996spectral}. 

The aim of this work is 
to explain why the numerical scheme based on the AW Hermite functions in the Petrov-Galerkin method is unstable and
to present a stable numerical scheme based on the AW Hermite functions in the Galerkin method,
where the AW Hermite functions are used as both trial and test functions.
Among recent works which explicitly mention the instability of 
the numerical scheme based on the AW Hermite functions 
in the Petrov-Galerkin method, 
we quote \cite{manzini2017convergence,funaro2021stability} where 
the first relies on adding a Fokker-Planck perturbation to enforce stability
while the second provides a mathematical investigation of the stability of the Hermite-Fourier spectral approximation of 
the VP model for a collisionless plasma in the electrostatic limit. 
The analysis includes high-order artificial collision operators of Lenard-Bernstein type. 
The next contribution, 
in \cite{bessemoulin2022stability,bessemoulin2023convergence}, 
evolves the AW Hermite functions with a time-dependent scaling 
following an idea originally introduced in \cite{ma2005hermite,ma2007stabilized} 
for second-order differential equations.
The authors introduce a time-dependent weighted norm 
whose evolution can be interpreted as an effective increase 
in the reference temperature over time.
A related approach is presented in \cite{pagliantini2023physics}, 
where the authors propose a spectral method for the 1D-1V VP system. 
The discretization in velocity space is based on AW Hermite functions, 
which are dynamically adapted via a scaling and shifting of the velocity variable, 
to maintain the stability of the numerical scheme.
In particular, at each time instant, an adaptivity criterion is used to select updated values of the scaling and shifting 
based on the numerical solution of the discrete VP system obtained at that time step.
Finally, \cite{dai2024quadratic} provides an analysis of 
the instability mechanisms associated with the use of 
AW Hermite functions for the discretization of the linear transport equation.

%
%

We next introduce the VP system and the numerical method, 
which is the model that we study in the present article.
This system describes the temporal evolution of the plasma particle distribution function
under the influence of a self-consistent electrostatic field,
generated by the charge and current densities of the particles themselves.
In this study, we concentrate on a dimensionless,
one-dimensional VP system for a single particle species
\begin{equation}\label{eq:VPSystem}
\left\{
\begin{aligned}
&   \frac{\partial f}{\partial t}
+ v \frac{\partial f}{\partial x}
+ E \frac{\partial f}{\partial v} = 0, \\
&   \frac{\partial E}{\partial x} = \rho - \rho_0, \\
\end{aligned}
\right.
\end{equation}
with \( t \geq 0 \), position \( x \in (0, L) \) and velocity \( v \in \mathbb{R} \).
The self-consistent electric field $E(t,x)$ is determined by the Poisson equation.
Periodic boundary conditions are prescribed in space.
The density \( \rho \) is given by
\begin{equation}
    \rho(t, x) = \int_{\mathbb{R}} f(t, x, v) \, dv.
\end{equation}
The constant \( \rho_0 \) ensures the quasi-neutrality condition of the plasma
\begin{equation}
    \int_0^L (\rho - \rho_0) \, dx = 0.
\end{equation}
In this work, the natural functional setting is
\[
   L_{\omega}^2
:= \left\{
   g : \mathbb{R} \times \mathbb{R} \to \mathbb{R}
     : \int_{\mathbb{R}^2} |g(x,v)|^2 \exp(\frac{v^2}{T}) 
     dxdv < +\infty
   \right\}.
\]
where $T>0$ denotes the reference temperature.
The corresponding norm is
\[
  \|g\|_{\omega}^2 
= \int_{\mathbb{R}^2} |g(x,v)|^2 \exp(\frac{v^2}{T}) dxdv. 
\]
We also consider the standard $L^2$ norm,
\[
  \|g\|^2 
= \int_{\mathbb{R}^2} |g(x,v)|^2 dxdv. 
\]
Since $\exp(\frac{v^2}{T})\geq 1$ for all $v\in\mathbb{R}$, 
it follows immediately that
\[
  \|f\|^2 \le \|f\|_{\omega}^2.
\]
A popular choice to approximate the distribution function $f(t,x,v)$ is to use a finite sum 
which corresponds to a truncation of a series
\begin{equation}\label{eq:truncation_series}
f_N(t,x,v) = \sum_{n=0}^N u_n(t,x) \psi_n(v),
\end{equation}
where $N+1$ is the number of modes, and $(\psi_n(v))_{n=0,\cdots,N}$ are AW Hermite functions
which are introduced in the next section.
We employ~\eqref{eq:truncation_series} in the velocity variable as both trial and test functions, 
thereby obtaning a Galerkin formulation.
In contrast to the Petrov-Galerkin method, the trial and test functions are identical and no longer orthogonal, 
which give rise to a Gram matrix.
Furthermore, we consider a spatial discretization for the semi-discreted VP system, 
written as an hyperbolic system: Fourier spectral method.
In \cite{manzini2017convergence}, a rigorous convergence theory for the SW Fourier–Hermite method is established.
However, there is no convergence theory for the AW Fourier-Hermite method because 
the method is not stable under the Petrov-Galerkin framework.
We propose the numerical method based on the AW Hermite functions in the Galerkin method, 
which is to be the stone to investigate the convergence analysis of the proposed method.

In Section~\ref{sec:Hermite spectral expansion of the Vlasov equation}, 
we introduce the formulation of the Vlasov equation
using AW Hermite functions in velocity within both the Petrov-Galerkin and Galerkin methods. 
We first examine the Petrov-Galerkin method and explain the source of its numerical instability.
We then introduce the Gram matrix arising from the Galerkin method.
Since the Gram matrix involved in the Galerkin method is dense,
its direct use leads to significant computational costs.
We then present an equivalent form to the Galerkin method, 
which maintains a modest computational.
Then in Section~\ref{sec:Fourier spectral method for the space discretization}, 
we introduce the Fourier spectral method for the spatial discretizations.
Finally in Section~\ref{sec:Numerical results}, 
we present numerical results for 
advection in velocity, two stream instability and bump-on-tail problem, 
highlighting the conservation properties and stability of the proposed discretization.

\section{Hermite spectral expansion of the Vlasov equation}
\label{sec:Hermite spectral expansion of the Vlasov equation}

For the solution of the VP system~\eqref{eq:VPSystem}, a corresponding variational problem can be derived, 
and the solution of the variational problem 
can be approximated by Galerkin methods, 
using AW Hermite functions as trial functions and test functions.
Petrov-Galerkin methods extend the idea of Galerkin methods 
using different functions for the approximate solution and the test functions.
Firstly, we introduce the AW Hermite functions. 

\subsection{Hermite functions}
We seek the approximation $f_N$ of the solution $f$ to the VP system \eqref{eq:VPSystem} 
by the finite sum \eqref{eq:truncation_series}.
Let \( ({H}_m)_{m \in \mathbb{N}} \)
be the family of Hermite polynomials \cite{nist},
\[
{H}_m(v) = (-1)^m e^{v^2} \frac{d^m}{dv^m}\left(e^{-v^2}\right),
\]
that is orthogonal on \( (-\infty, \infty) \) with respect to the
Gaussian function $\exp(-v^2)$.
We choose the following basis of normalized scaled time-independent AW Hermite functions \( (\psi_m)_{m \in \mathbb{N}} \):
\begin{equation}\label{eq:hermite_functions}
  \psi_m(v)
= \left( 2^m m! \sqrt{\pi T} \right)^{-\frac12}
  {H}_m\left(\frac{v}{\sqrt{T}}\right) 
  \exp{(-\frac{v^2}{T})},
  \quad m \geq 0,
\end{equation}
The functions \((\psi_m)_{m \in \mathbb{N}}\) satisfy the following recurrence relations:
\begin{equation}\label{eq:recurrence_relation}
\begin{aligned}
  \dfrac{ \partial \psi_{m}(v) }{\partial v}
&=-\sqrt{ \frac{2}T} 
  \sqrt{m+1} \psi_{m+1}; 
\\
  v \, \psi_{m}(v) 
&= \sqrt{ \frac{T}2} 
  \left( \sqrt{m  } \psi_{m-1} 
       + \sqrt{m+1} \psi_{m+1}
  \right).
\end{aligned}
\end{equation}
To obtain an orthogonal system, 
we choose the following weight function
\[
\omega(v) = \exp{(\frac{v^2}{T})}, \quad T > 0,
\]
such that he AW Hermite functions \eqref{eq:hermite_functions} satisfy the following orthogonality
\begin{equation}\label{eq:orthogonality_property}
\int_{\mathbb{R}} \psi_n(v) \psi_m(v) \omega(v) dv = \delta_{nm}.
\end{equation}

\subsection{Petrov-Galerkin method}

We take $\psi_m(v)$ as trial function and $\psi_n(v) \omega(v)$ as test function.
Then applying the Petrov-Galerkin method, 
the Hermite spectral form of the Vlasov equation is obtained using
the orthogonality property \eqref{eq:orthogonality_property}
and the properties of Hermite functions \eqref{eq:recurrence_relation}.
In the expansion~\eqref{eq:truncation_series}, we take $N=\infty$. 
Inserting the expansion~\eqref{eq:truncation_series},
we compute the different terms of the Vlasov equation in~\eqref{eq:VPSystem}.
Since the AW Hermite functions are not dependent on the variable $t$, 
the time derivative term is simply given by
\[
\dfrac{\partial f}{\partial t} = \sum_{m=0}^\infty \partial_t u_m \psi_m.
\]
The transport term is
\[
  v \dfrac{\partial f}{\partial x}
= \sum_{m=0}^\infty \sqrt{\frac{T}2} \partial_x u_{m}
  \left( \sqrt{m+1} \psi_{m+1}
       + \sqrt{m}   \psi_{m-1} \right), 
\]
and the nonlinear term is
\[
  E \dfrac{\partial f}{\partial v}
= -\sum_{m=0}^\infty \sqrt{\frac{2}T} E \sqrt{m}u_{m-1}
\]
Then, we arrive at the following form, for any $t \geq 0$,
\[
\begin{aligned}
\partial_t u_m
&+ \sqrt{\frac{T}2}
  \left( \sqrt{m+1} \partial_x  u_{m+1}
       + \sqrt{m}   \partial_x  u_{m-1} \right) \\
&- \sqrt{\frac{2}T} E \sqrt{m}u_{m-1}
= 0,
\quad m \geq 0,
\end{aligned}
\]
with the understanding that $u_m = 0$ for $m < 0$.
This is an infinite system for $(u_n)_{n\in\mathbb{N}}$ and $E$.
The system can also be written in matrix form as
\begin{equation}\label{eq:petrov-galerkin-matrix-form}
  \partial_t U
+ \sqrt{\frac{T}2} B \partial_x U
- \sqrt{\frac{2}T} E D U 
= 0,
\end{equation}
where we define the infinite triangular and sparse matrices $B, D \in \mathbb{R}^{\mathbb{N}\times\mathbb{N}}$
\[
\begin{aligned}
&B = (b_{mn})_{m,n\geq 0}, \quad b_{mn} = \sqrt{m}\,\delta_{m-1,n} + \sqrt{m+1}\,\delta_{m+1,n}, \\
&D = (d_{mn})_{m,n\geq 0}, \quad d_{mn} = \sqrt{m}\,\delta_{m-1,n}.
\end{aligned}
\]
And the infinite vector of unknowns is $U(t,x) = (u_{m}(t,x))_{m\geq 0} \in \mathbb{R}^{\mathbb{N}}$.
The infinite system for $U(t,x)$ and $E(t,x)$ is formally 
equivalent to the VP system~\eqref{eq:VPSystem}.

In the following, we analyze the truncated system and introduce a block decomposition of the infinite matrices and vectors involved. 
For notational convenience, we first introduce the following definitions.
\begin{defi}[Block matrix by truncation $N$]
A block matrix by truncation $N$ of a infinite matrix $Q \in \mathbb{R}^{\mathbb{N}\times\mathbb{N}}$,
which is partitioned into a collection of $2 \times 2$ smaller matrices
$$
\renewcommand{\arraystretch}{1.5}
Q =
\begin{bmatrix}
Q_{11}^N & Q_{12}^N \\
Q_{21}^N & Q_{22}^N \\
\end{bmatrix},
$$
where
$$
\begin{aligned}
Q_{11}^N &= (q_{nm})_{0\leq n,m \leq N}, \\
Q_{12}^N &= (q_{nm})_{0\leq n \leq N, \,N+1\leq m \leq \infty}, \\
Q_{21}^N &= (q_{nm})_{N+1\leq n \leq \infty, \,0\leq m \leq N}, \\
Q_{22}^N &= (q_{nm})_{N+1\leq n,m \leq \infty}. \\
\end{aligned}
$$
Unless otherwise specified, 
all block matrices $Q_{11}, Q_{12}, Q_{21}, Q_{22}$ 
are understood to be associated with the truncation $N$, 
and the superscript $\cdot^N$ is omitted throughout to simplify the notation.
\end{defi}
\begin{defi}[Block vector by truncation $N$]
A block vector by truncation $N$ of a infinite vector $X \in \mathbb{R}^{\mathbb{N}}$,
which is partitioned into a collection of $2$ smaller vectors 
$$
\renewcommand{\arraystretch}{1.5}
X =
\begin{bmatrix}
X_{1}^N \\
X_{2}^N \\
\end{bmatrix},
$$
where
$$
\begin{aligned}
X_{1}^N &= (x_{n})_{0\leq n \leq N}, \\
X_{2}^N &= (x_{n})_{N+1\leq n \leq \infty}, \\
\end{aligned}
$$
Unless otherwise specified, 
all block vectors $X_{1}, X_{2}$ 
are understood to be associated with the truncation $N$, 
and the superscript $\cdot^N$ is omitted throughout to simplify the notation.
\end{defi}

\subsection{Origin of the numerical instability}
For the sake of simplicity,  we set the reference temperature to \(T=1\).
Then we take the block matrix $B$, $D$ and the block vector $U$ 
and further assume that $U_2$ is null vector.
Applying the Petrov-Galerkin method leads to the equations of the Vlasov equation
\[
   \dfrac{\partial U_1}{\partial t} 
+ \sqrt{\frac{T}2} B_{11}\dfrac{\partial U_1}{\partial x}
- \sqrt{\frac{2}T} E(t,x) D_{11} U_1
=0.
\]
We now consider the case of a constant electric field, $E(t,x)=-\sqrt{\frac{T}2}$
and prescribe the initial condition $U_1(0,x) = (1, 0, 0,\dots)^\top$.
Since the initial condition is independent of \(x\), 
the solution remains spatially homogeneous. 
Consequently, the transport term involving 
the symmetric matrix \(B_{11}\) vanishes, 
and the system reduces to the ordinary differential equation
\[
   \dfrac{\partial U_1}{\partial t} 
+ D_{11} U_1
=0.
\]
Since the matrix $D_{11}$ is nilpotent, 
so the solution $U_1(t) = e^{-t D_{11}}\,U_1(0)$ is a finite polynomial in $t$,
which gives
\[
u_n(t) = (-1)^n \dfrac{t^n}{\sqrt{n!}}.
\]
So the distribution function $f$ is given by
\[
f_N(t,v) = \sum_{n=0}^N (-1)^n \dfrac{t^n}{\sqrt{n!}} \psi_n(v).
\]
If $N$ is finite, we have 
\[
\|U_1(t)\|_{L_\omega^2}^2 = \sum_{n=0}^N \dfrac{t^{2n}}{n!}.
\]
For large $t$, we have 
\begin{equation}\label{eqn:norm_truncated_moments}
\|U_1(t)\|_{L_\omega^2}^2 = \dfrac{|t|^{N}}{\sqrt{N!}}.
\end{equation}
Thus the numerical solution $f_N$ blows up. 
If $N=\infty$, we have
\[
f_N(t,v) = \sum_{n=0}^\infty (-1)^n \dfrac{t^n}{\sqrt{n!}} \psi_n(v).
\]
Inserting the basis functions
\[
\psi_n(v) = \dfrac{1}{\sqrt{2^n n!}} {H}_n(v) \dfrac{1}{\sqrt{\pi}} e^{-v^2}
\]
into the infinite series, we obtain
\[
\begin{aligned}
   f(t,v) 
&= \dfrac{e^{-v^2}}{\sqrt{\pi}} 
   \sum_{n=0}^\infty \dfrac{(-1)^n t^n}{\sqrt{n!}} \dfrac{H_n(v)}{\sqrt{2^n n!}} \\
&= \dfrac{e^{-v^2}}{\sqrt{\pi}}
   \sum_{n=0}^\infty \dfrac{H_n(v)}{n!} \left( \dfrac{-t}{\sqrt{2}} \right)^2.
\end{aligned}
\]
Using the generating formula 
\[
\sum_{n=0}^\infty \dfrac{H_n(v)}{n!} s^n = e^{2vs -s^2},
\]
we obtain
\[
  f(t,v) 
= \dfrac{1}{\sqrt{\pi}} \exp{\left(-\left(v+\dfrac{t}{\sqrt2}\right)^2\right)}. 
\]
We complicated infinite-dimensional ODE system simply corresponds to 
a Gaussian translating in velocity $v$.
Specifically,
\begin{equation}\label{eqn:solution_boundedness}
f(t,v) = f_0\left(\dfrac{t}{\sqrt2} + v\right).
\end{equation}
Since the initial condition is $f(0, v) = \dfrac{1}{\sqrt\pi} e^{-v^2}$, 
the evolution is a rigid shift to the left with speed $1/\sqrt2$.

Equation~\eqref{eqn:norm_truncated_moments} 
shows that truncating the Hermite expansion to a finite number of moments
induces numerical instability. 
As a consequence, the bounded solution~\eqref{eqn:solution_boundedness} 
is no longer preserved by the truncated moment system.
This observation suggests that the numerical instability 
originates from the truncation of the matrix \(D\).

\subsection{Gram matrix}

We introduce the following definition which is yield by 
the collection of all scalar products of the AW Hermite functions.
\begin{defi}[Gram matrix]
The infinite symmetric Gram matrix 
$A = A^{\top} = (a_{mn})_{m,n \geq 0} \in \mathbb{R}^{\mathbb{N} \times \mathbb{N}}$ of 
the problem is the collection of all scalar products of the AW 
Hermite functions. The coefficient is 
$$
a_{mn} = \int_{\mathbb{R}} \psi_m(v) \psi_n(v) dv.
$$
\end{defi}
For simplicity, we set $T=2$ throughout this subsection.
Assume that the initial condition satisfies $f(0)\in L_\omega^2$. 
Since $L_\omega^2\subset L^2$, 
it follows that $f(0)\in L^2$. 
Any solution of the equation~\eqref{eq:petrov-galerkin-matrix-form} preserves the $L^2$-norm, namely,
\[
\frac{d}{dt}\|f(t)\|_{L^2}^2 = 0.
\]
Indeed, this follows from the identity
\begin{equation}\label{eq:skew-symmetric-identity}
\begin{aligned}
   \frac12  \frac d{dt} \|f(t)\|_{L^2}^2  
&= \frac12  \frac d{dt} \int_0^L \left< U, AU \right> dx \\
&= \int_0^L \left< U, A \frac \partial {\partial t} U \right> dx \\
&=-\int_0^L \left< U, AB\frac \partial {\partial x} U \right> dx 
  +\int_0^L \left< U, EADU \right> dx \\
&=0,
\end{aligned}
\end{equation}
where the last equality follows from the skew-symmetry relation
\(
AD + D^\top A=0,
\)
proved in \cite[Lemma 3.3]{dai2024quadratic}, 
together with periodic boundary conditions. 
Here, \(\langle\cdot,\cdot\rangle\) denotes the standard 
Euclidean inner product on the corresponding vector space.

For the truncated system of~\eqref{eq:petrov-galerkin-matrix-form}
\[
\begin{aligned}
   \frac12  \frac d{dt} \|f_N(t)\|_{L^2}^2  
=& \frac12  \frac d{dt} \int_0^L \left< U_1, A_{11}U_1 \right> dx \\
=& \int_0^L \left< U_1, A_{11} \frac \partial {\partial t} U_1 \right> dx \\
=&-\int_0^L \left< U_1, A_{11}B_{11}
   \frac \partial {\partial x} U_1 \right> dx \\ 
 &+\int_0^L \left< U_1, E_N A_{11}D_{11}
   U_1 \right> dx \\
\neq&\,0,
\end{aligned}
\]
since \(A_{11}D_{11} + D_{11}^\top A_{11} \neq 0\) 
by~\cite[Lemma 3.4]{dai2024quadratic}.
Hence, the truncation to a finite number of moments 
does not preserve the skew-symmetric structure 
shown in~\eqref{eq:skew-symmetric-identity}. 
Consequently, the approximation of the VP system 
no longer satisfies the preservation 
of the energy~\eqref{eq:skew-symmetric-identity}.

\subsection{Galerkin method}
To apply the Galerkin method to the Valsov equation, 
we take $\psi_m(v)$ as trial and test functions.
Contrary to the Petrov-Galerkin method, 
the orthogonality property \eqref{eq:orthogonality_property} 
is not available anymore
to the Hermite spectral form of the Valsov equation.
As a result, a nontrivial Gram matrix arises.

We derive a new Hermite spectral form of the Vlasov equation 
using the Galerkin method. 
We obtain a new evolution equation for $u_m$, $m \in \mathbb{N}$:
\begin{equation}\label{eq:modelFaible}
\begin{aligned}
   \sum_{m=0}^{\infty} a_{nm} \partial_t u_m 
+& \sqrt{\frac{T}2}
   \sum_{m=0}^{\infty} 
   \left( \sqrt{m}  a_{n,m-1}
	+ \sqrt{m+1}a_{n,m+1}
   \right) \partial_x u_m \\
-& \sqrt{ \frac{2}T} E
   \sum_{m=0}^{\infty} \sqrt{m} a_{n,m-1} u_m
=0.
\end{aligned}
\end{equation}
Here, we take $N=\infty$ in the expansion~\eqref{eq:modelFaible}, 
which is formally equivalent to the Vlasov equation of the system~\eqref{eq:VPSystem}.
The system can also be written in matrix form as
\[
  A \partial_t U
+ \sqrt{\frac{T}2} A B \partial_x U
- \sqrt{\frac{2}T} E A D U
= 0.
\]
If we take $N$ a finite number, we obtain a truncated system as follow
\begin{equation}\label{eq:galerkinMethodTruncated}
\begin{aligned}
   \sum_{m=0}^{N} a_{nm} \partial_t u_m 
+& \sqrt{\frac{T}2}
   \sum_{m=0}^{N} 
   \left( \sqrt{m}  a_{n,m-1}
	+ \sqrt{m+1}a_{n,m+1}
   \right) \partial_x u_m \\
-& \sqrt{ \frac{2}T} E_N
   \sum_{m=0}^{N} \sqrt{m} a_{n,m-1} u_m
=0.
\end{aligned}
\end{equation}
The self-consistent electric field $E_N$ is determined by the Poisson equation.
The system can be rewritten in matrix form 
\[
\begin{aligned}
 A_{11} \partial_t U_1
&+ \sqrt{\frac{T}2}     \left( A_{11}B_{11} + A_{12}B_{21} \right) \partial_x U_1 \\
&- \sqrt{\frac{2}T} E_N \left( A_{11}D_{11} + A_{12}D_{21} \right) U_1
 = 0.
\end{aligned}
\]
However, since \(A_{11}\) is a dense matrix, this formulation leads to a significantly higher computational cost.

\subsection{Equivalent form}
In what follows, we address this computational challenge.
\begin{defi}[Gram kernel matrix]\label{defi:gram_kernel_matrix}
The infinite Gram kernel matrix 
$Z = (z_{mn})_{m,n \geq 0} \in \mathbb{R}^{\mathbb{N} \times \mathbb{N}}$ of the problem 
is an upper triangular matrix of which the entries are
\[
\left\{
\begin{aligned}
&  z_{mn} = \sqrt{\dfrac{n!}{m!}} 
   \dfrac{1}{4^{\frac{n-m}2}(\frac{n-m}2)!},
&& n-m \in 2\mathbb{N},  \\
&  z_{mn} = 0,
&& \text{otherwise}.\\
\end{aligned}
\right.
\]
\end{defi}
\begin{Remark}
The entries on the main diagonal of the Gram kernel matrix are equal to $1$.
\end{Remark}
\begin{proposition}\label{prop:relation_interesting}
One has
\begin{equation}\label{eq:relationInteresting}
\sum_{m=0}^{N+1} a_{nm}\, z_{m,N+1}   = 0, \quad \, 0\leq n \leq N, 
\end{equation}
where $a_{nm}, z_{m,N+1}$ are the entriex of the Gram matrix and the Gram kernel matrix, respectively.
\end{proposition}
Substituting the relation~\eqref{eq:relationInteresting} to 
the evolution equation~\eqref{eq:modelFaible},
we obtain
\[
\begin{aligned}
 & \sum_{m=0}^{N} a_{nm} \partial_t u_m \\
+& \sqrt{\frac{T}2} \sum_{m=0}^{N} a_{nm}
   \left( \sqrt{m}  \partial_x u_{m-1}
        + \sqrt{m+1}\partial_x u_{m+1}
	- z_{m,N+1} \sqrt{N+1} \partial_x u_N   
   \right) \\ 
-& \sqrt{\frac{2}T} E_N \sum_{m=0}^{N} a_{nm}
   \left( \sqrt{m}  u_{m-1} 
	- z_{m,N+1} \sqrt{N+1} u_N  
   \right) = 0. \\
\end{aligned}
\]
Equivalently:
\begin{equation}\label{eq:method_galerkin_vlasov_velocity_space}
\begin{aligned}
   \partial_t u_m 
+& \sqrt{\frac{T}2}
   \left( \sqrt{m}        \partial_x u_{m-1}
	+ \sqrt{m+1}      \partial_x u_{m+1}
	- z_{m,N+1} \sqrt{N+1}\partial_x u_N 
   \right) \\
-& \sqrt{\frac{2}T} E_N 
   \left( \sqrt{m} u_{m-1}
	- z_{m,N+1} \sqrt{N+1} u_N 
   \right)
= 0, \quad m = 0, \ldots, N,\\
\end{aligned}
\end{equation}
with the convention that $u_n = 0$ for $n<0$ and $n>N$.
For the Poisson equation, we note that the density $\rho_N$ satisfies 
\[
\rho_N = \int_{\mathbb{R}} f_N dv = u_0.
\]
Then we have
\begin{equation}\label{eq:method_galerkin_poisson_velocity_space}
\dfrac{\partial E_N}{ \partial x} = u_0 - \rho_0.
\end{equation}
\begin{Remark}
Writing the system in matrix form, we have
\[
  \partial_t U_1
+ \sqrt{\frac{T}2}     \overline{B}_{11} \partial_x U_1
- \sqrt{\frac{2}T} E_N \overline{D}_{11} U_1
= 0,
\]
where
\[
\begin{aligned}
\overline{B}_{11} &= B_{11} + A_{11}^{-1} A_{12} B_{21}, \\
\overline{D}_{11} &= D_{11} + A_{11}^{-1} A_{12} D_{21}. \\
\end{aligned}
\]
The terms \(A_{11}^{-1} A_{12} B_{21}\) 
and \(A_{11}^{-1} A_{12} D_{21}\) can be 
expressed by \( (z_{m, N+1})_{0\le m \le N+1} \) 
given by Proposition~\ref{prop:relation_interesting}. 
For example, 
if \(N\) is even, the matrices \(\overline{B}_{11}\) and \(\overline{D}_{11}\) are given by
\[
\renewcommand{\arraystretch}{1.8}
\overline{B}_{11}
=
\begin{bmatrix}
0      & \sqrt1 &        &            &            & \\
\sqrt1 & 0      & \sqrt2 &            &            & \sqrt{N+1} z_{1,N+1} \\
       & \sqrt2 & 0      & \ddots     &            & \\
       &        & \ddots & 0          & \sqrt{N-1} & \\
       &        &        & \sqrt{N-1} & 0          & \sqrt{N} + \sqrt{N+1} z_{N-1,N+1} \\
       &        &        &            & \sqrt{N}   & 0 \\
\end{bmatrix},
\]
and
\[
\renewcommand{\arraystretch}{1.8}
\overline{D}_{11}
=
\begin{bmatrix}
0      &        &        &            &          & \\
\sqrt1 & 0      &        &            &          & \sqrt{N+1} z_{1,N+1} \\
       & \sqrt2 & 0      &            &          & \\
       &        & \ddots & 0          &          & \\
       &        &        & \sqrt{N-1} & 0        & \sqrt{N+1} z_{N-1,N+1} \\
       &        &        &            & \sqrt{N} & 0 \\
\end{bmatrix}.
\]
If \(N\) is odd, the matrices \(\overline{B}_{11}\) and \(\overline{D}_{11}\) have the same structure, 
except that in the last column the nonzero entries occur at even-indexed rows, while the odd-indexed rows are zero.

From a numerical standpoint, the Galerkin method introduces only a single additional column into the matrix. 
Consequently, the resulting increase in computational cost remains modest and compares favorably 
with that of the Petrov–Galerkin method.
\end{Remark}

\section{Fourier spectral method for the space discretization}
\label{sec:Fourier spectral method for the space discretization}
We consider a Fourier spectral discretization in space combined with
AW Hermite functions in velocity.
Our goal is to find an approximation $f_{N,J}$ of the distribution
function $f$, defined by
\begin{equation}\label{eq:fourierHermiteExpansion}
f_{N,J}(t,x,v)
=
\sum_{m=0}^{N}
\sum_{j=-J}^{J}
c_{m,j}(t)\,\xi_j(x)\,\psi_m(v),
\end{equation}
where
\begin{equation}
   \xi_j(x)
:= \frac{1}{\sqrt L} e^{2 \imath \pi j x/L},
   \qquad j=-J,\ldots,J,
\end{equation}
and $(\psi_m)_{m\ge0}$ denotes the family of AW Hermite functions.
This corresponds to approximate the Hermite modes $u_{m}(t,x)$ as
\[
u_{m,J}(t,x) = \sum_{j=-J}^{J} c_{m,j}(t)\,\xi_j(x).
\]
The electric field $E_{N}$ is similarly approximated in the Fourier basis as:
\begin{equation}\label{eq:electricFieldFourier}
E_{N,J}(t, x) = \sum_{j=-J}^{J} e_j(t) \xi_j(x).
\end{equation}
We derive equations governing the coefficients $c_{m,j}$ and $e_j$
by examining the weak formulation for $u_{m,J}$ and $E_{N,J}$ in the space
\[
V_{J} := \text{span} \{ \xi_j, j = -J, \ldots, J \},
\]
and taking $(\xi_j)_{-J\leq j\leq J}$ as test functions.

To finalize the method's formulation, we need to compute
the coefficients of the electric field in the Fourier basis.
Using the representation of the distribution function in the Fourier-Hermite
discretization, we obtain
\[
\rho_{N,J}(t,x) = \sum_{m=0}^N \sum_{j=-J}^J c_{m,j} \xi_j(x).
\]
We need to consider the potential function $\Phi_{N,J}(t,x)$ such that
\[
\left\{
\begin{aligned}
E_{N,J} &= -\dfrac{\partial \Phi_{N,J}}{\partial x} , \\
\dfrac{\partial E_{N,J}}{\partial x} &= \rho_{N,J} - \rho_0 .
\end{aligned}
\right.
\]
Hence, we get the one dimensional Poisson equation
\[
-\dfrac{\partial^2 \Phi_{N,J}}{\partial x^2} = \rho_{N,J} - \rho_0 .
\]
Consider the Galerkin method for the Poisson equation
\[
  \int_0^L -\dfrac{\partial^2 \Phi_{N,J}}{\partial x^2} \eta
= \int_0^L (\rho_{N,J} - \rho_0) \eta,
\]
which gives the following Fourier representation of the electric field
\begin{equation}\label{eq:electric_field_modes}
e_j = \left\{
\begin{aligned}
&0, &&j = 0, \\
&\dfrac{-\imath L}{2\pi j} \sum_{m=0}^N c_{m,j} , &&j \neq 0.
\end{aligned}
\right.
\end{equation}
We can also get the Fourier representation of
$\Phi_{N,J} = \sum_{j=-J}^{J} \Phi^j(t) \xi_j(x)$
\[
\Phi^j = \left\{
\begin{aligned}
&0, &&j = 0, \\
&
\left( \dfrac{L}{2\pi j} \right)^2
\sum_{m=0}^N c_{m,j} , &&j \neq 0.
\end{aligned}
\right.
\]
With the Galerkin method in Fourier space,
we get the following system of PDEs for the coefficients:
\begin{equation}\label{eq:galerkin-fourier-hermite}
\begin{aligned}
   \partial_t c_{m,j}
& -\sqrt{\frac{T}2} \frac{2\pi j \imath}{L}
   \left( \sqrt{m}   c_{m-1,j}
	+ \sqrt{m+1} c_{m+1,j}
	- z_{m,N+1} \sqrt{N+1} c_{N,j}
   \right) \\
& -\sqrt{\frac{2}T}
  \left(  \sqrt{m}   [E \ast c_{m-1}][j]
	- z_{m,N+1} \sqrt{N+1} [E \ast c_{N}  ][j]
  \right) = 0,
\end{aligned}
\end{equation}
with the convolution
\[
[E \ast c_m][p] = \frac{1}{\sqrt{L}} \sum_{j=-J}^J e_{p-j} c_{m,j}.
\]

\section{Numerical results}
\label{sec:Numerical results}

We implement a research code in Python to evaluate the Garlekin method 
and compare to the Petrov-Garlekin method. 
We apply a fourth order Runge-Kutta scheme to the Galerkin method with AW and Fourier functions for the Vlasov equation
with Fourier Galerkin approximation of the Poisson equation. 

\subsection{Advection in velocity}

\begin{figure}[h!]
\centering
\begin{tabular}{cc}
  \includegraphics[scale = 0.28]{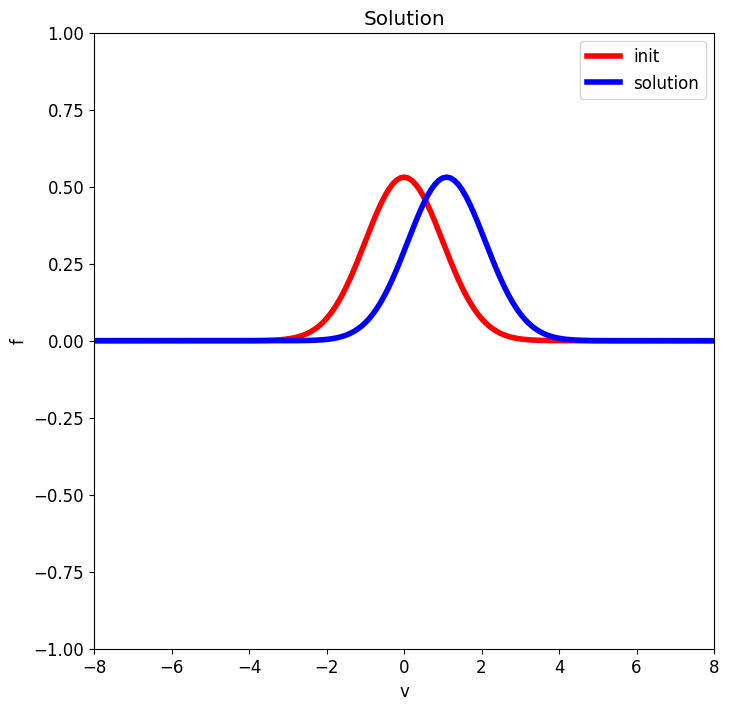} 
& \includegraphics[scale = 0.28]{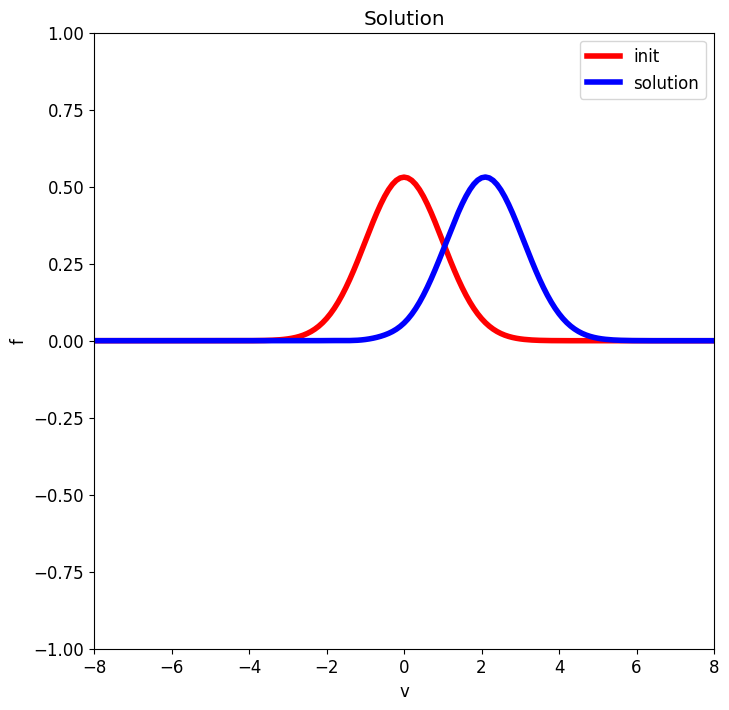} \\
  \includegraphics[scale = 0.28]{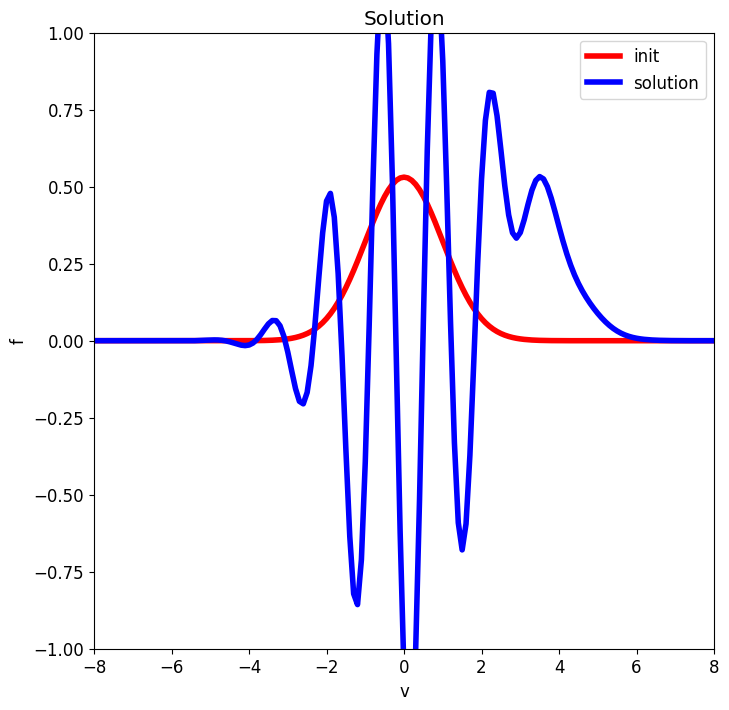}
& \includegraphics[scale = 0.28]{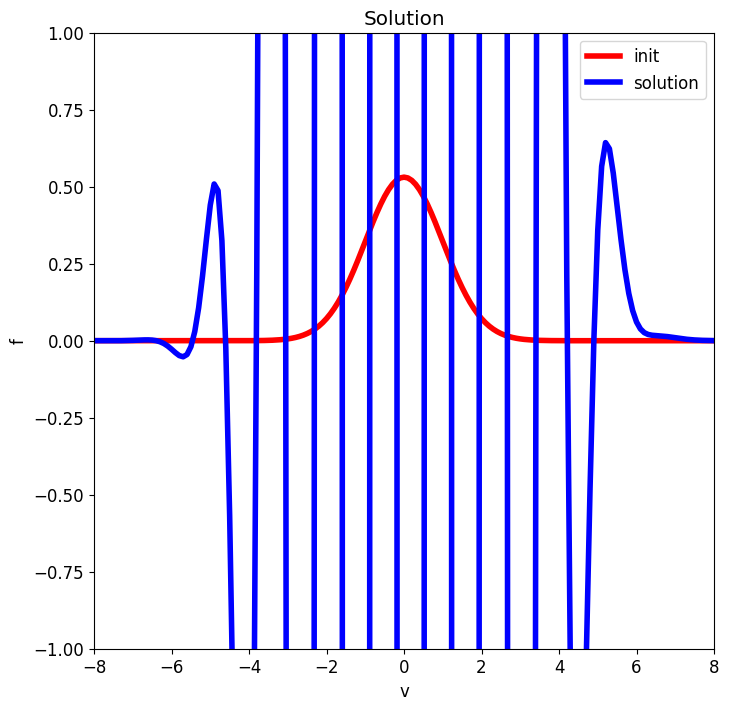}
\end{tabular}
\caption{Results of the advection test of the Petrov-Galerkin method at time $t_1=1$, $t_2=2$, $t_3=3$ then $t_4=4$ ($N=19$ and $\Delta t=0.1$).}
\label{fig:1}
\end{figure}
We consider a simplified case - the advection equation with a constant electric field ($E=-\sqrt{\frac{T}2}$) - to provide a simple example of the instability attached to the Petrov-Garlekin method. Therefore, the Vlasov equation is reduced to 
\[
\partial_t f + \partial_v f = 0.
\]
The discretized system in velocity of the Petrov-Galerkin method is reduced to
\[
  \partial_t u_m
- \sqrt{\frac{2}T} \sqrt{m}u_{m-1}
= 0,
\quad 0 \leq m \leq N.
\]
And the discretized system in velocity of the Galerkin method~\eqref{eq:method_galerkin_vlasov_velocity_space} is reduced to
\[
  \partial_t u_m 
- \sqrt{\frac{2}T} 
  \left( \sqrt{m} u_{m-1}
	+ z_{m,N+1} \sqrt{N+1} u_N 
  \right)
= 0, \quad m = 0, \ldots, N.
\]
The initial data, which is $U=(1,0, 0, \dots)$ and only the first moment is non zero, is a pure Gaussian.
An example of a simulation of the Petrov-Galerkin method is provided 
in Figure \ref{fig:1} at four different time $t_1=2$, $t_2=3$, $t_3=4$ and $t_4=5$. 
Until time $t\approx t_2$, the solution is correct. 
Then a numerical instability starts to be visible for $t\approx t_3$, 
and  blows up exponentially for $t\geq  t_4$.

\begin{figure}[H]
\centering
\begin{tabular}{cc}
  \includegraphics[scale = 0.28]{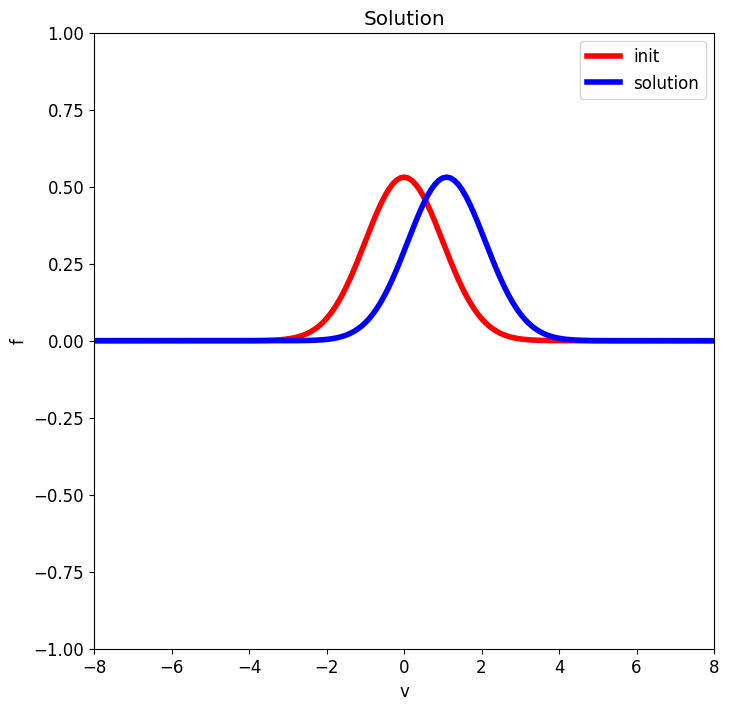}
& \includegraphics[scale = 0.28]{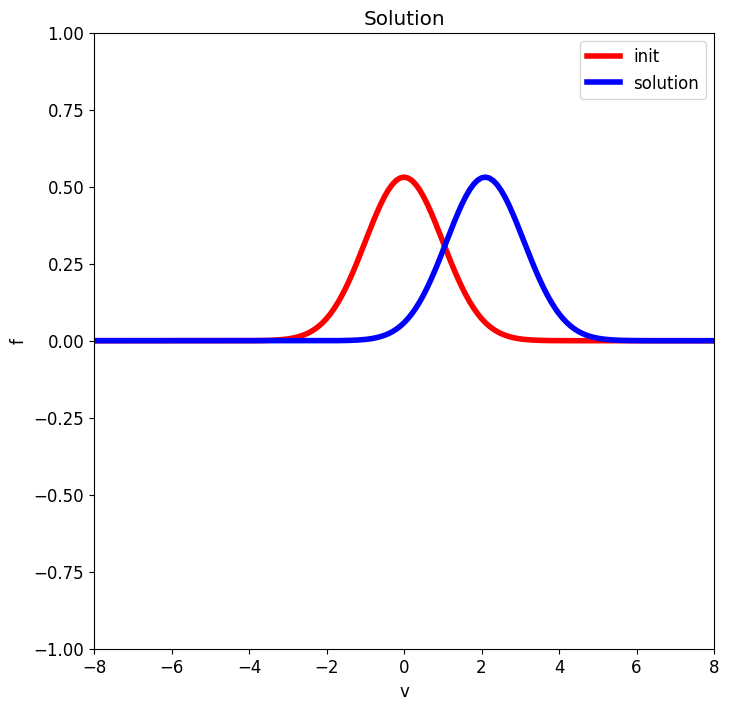} \\
  \includegraphics[scale = 0.28]{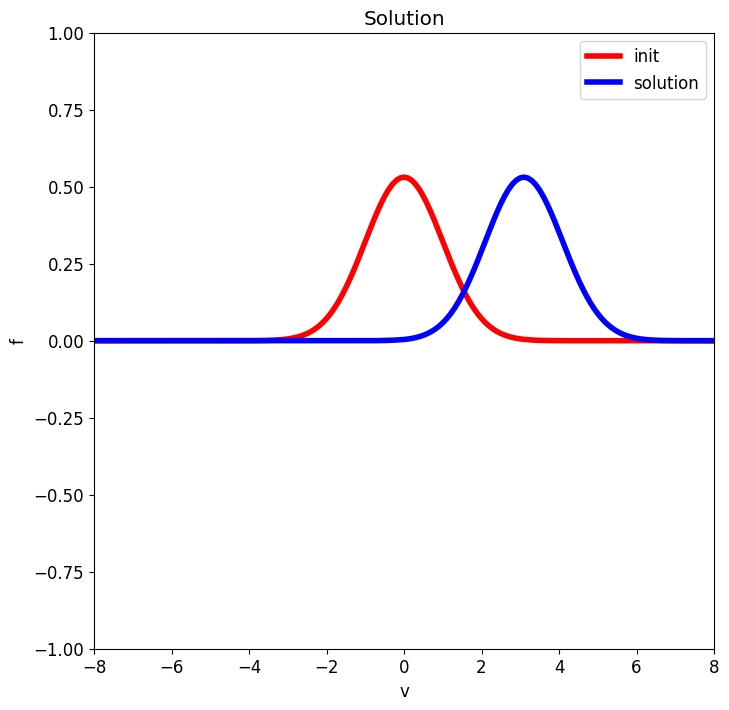}
& \includegraphics[scale = 0.28]{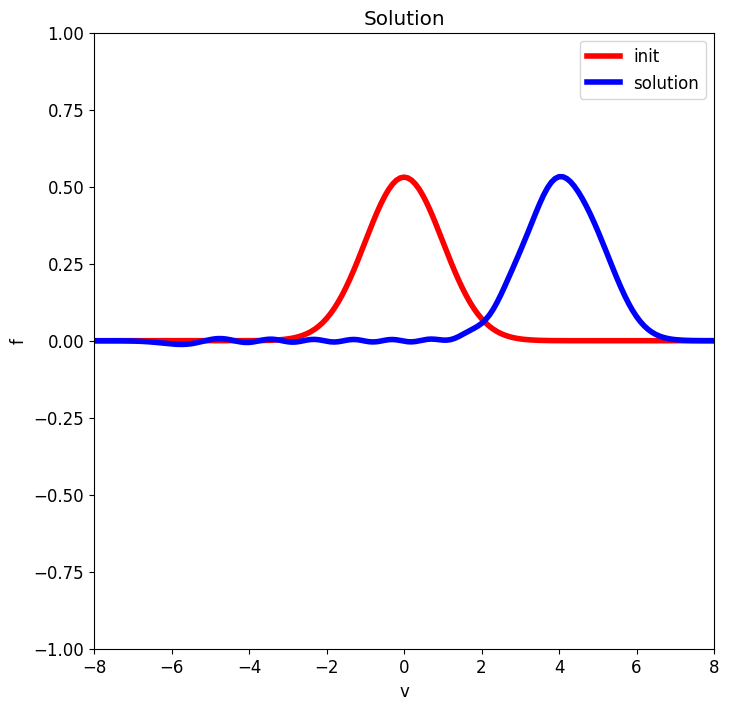}
\end{tabular}
\caption{Results of the advection test the Galerkin method at time $t_1=1$ to $t_4=4$ ($N=19$ and $\Delta t=0.1$).}
\label{fig:2}
\end{figure}
It is clear on the final result that the numerical simulation is spoiled with an important numerical instability 
which is in clear contradiction with the preservation of the quadratic norm.
The instability visible in Figure \ref{fig:1} is a paradox since the initial equation is stable.
The blow-up can be understood from Eq.~\eqref{eqn:norm_truncated_moments}: as t increases, the norm of the truncated moments grows exponentially, leading to a rapid blow-up.

Next, we recalculate the test of Figure~\ref{fig:1} using the Galerkin method with $N=19$, $\Delta t=0.1$ and $T=2$.
A numerical recurrence phenomenon~\cite{mehrenberger2020recurrence} is visible in Figure~\ref{fig:2}. 
The norm of the truncated moments is rigorously constant one time step after the other 
and the proposed method is stable.

\subsection{Two-stream instability}
We take the data of the two stream instability from \cite{filbet2022conservative}.
The initial data is
\[
f_0(x,v)=\frac27 \left( 1+\cos kx + \alpha (\cos 2kx +\cos 3kx) /1.2 \right)(1+v^2) \frac1{\sqrt{2\pi }} e^{-v^2/2}
\]
with $\alpha=0.01$ and $k=0.5$.
We perform tests at $N = 64$. For Fourier spectral discretization, we take $J=16$.
And the time step is $\Delta t = 0.01$.
For this problem only two moments are non zero, which are
\[
\left\{
\begin{aligned}
&u_0(x)= \frac{12}        7\left( 1+\cos kx + \alpha (\cos 2kx +\cos 3kx) /1.2 \right) \\
&u_2(x)= \frac{10\sqrt 2} 7\left( 1+\cos kx + \alpha (\cos 2kx +\cos 3kx) /1.2 \right),
\end{aligned}
\right.
\]
and all other moments vanish.
\begin{figure}[h!]
\centering
\begin{subfigure}[b]{0.45\textwidth}
  \centering
  \includegraphics[width=\textwidth]{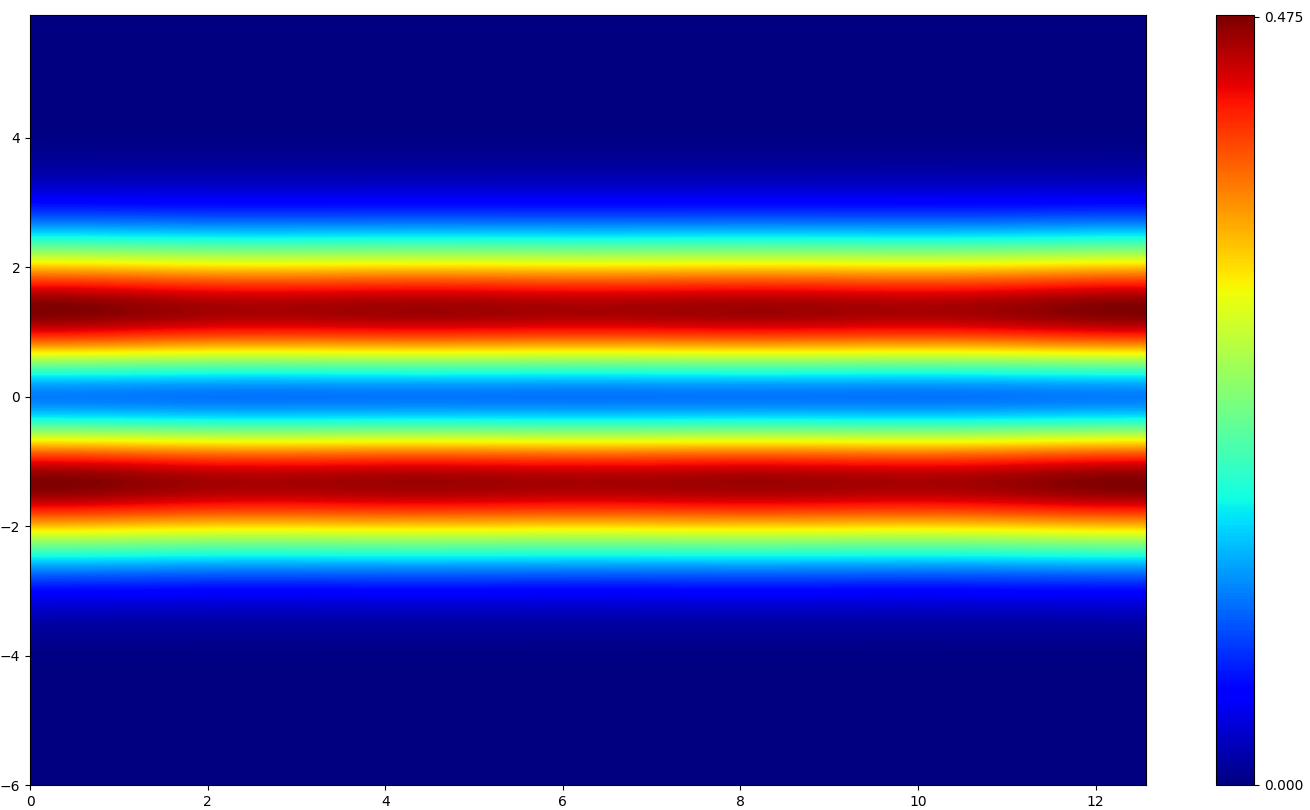}
  \caption{$t=0$}
\end{subfigure}
\hfill
\begin{subfigure}[b]{0.45\textwidth}
  \centering
  \includegraphics[width=\textwidth]{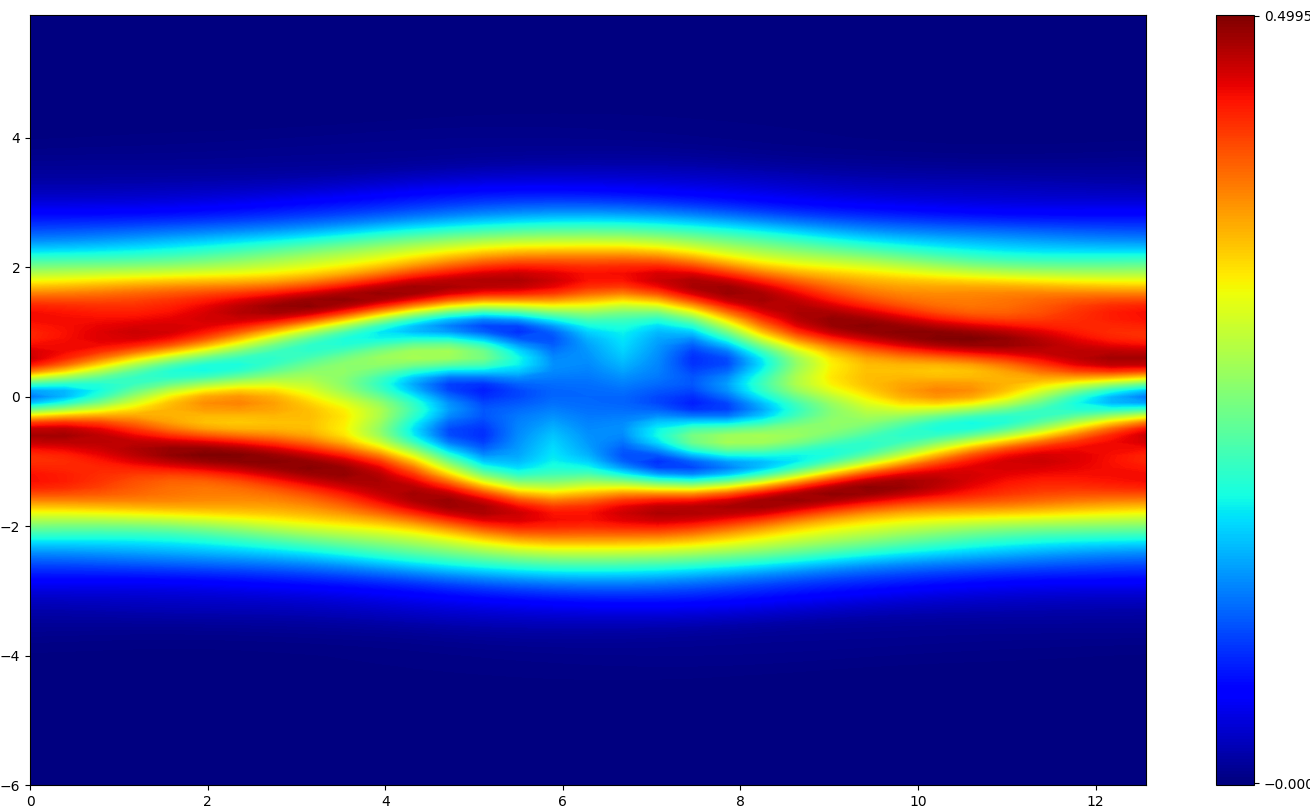}
  \caption{$t=20$}
\end{subfigure}
\caption{Density function for the two-stream instability computed with the Petrov-Galerkin method at times $t=0$ and $t=20$.}
\label{fig:two_stream_pertrov_galerkin}
\end{figure}
\begin{figure}[h!]
\centering
\begin{subfigure}[b]{0.45\textwidth}
  \centering
  \includegraphics[width=\textwidth]{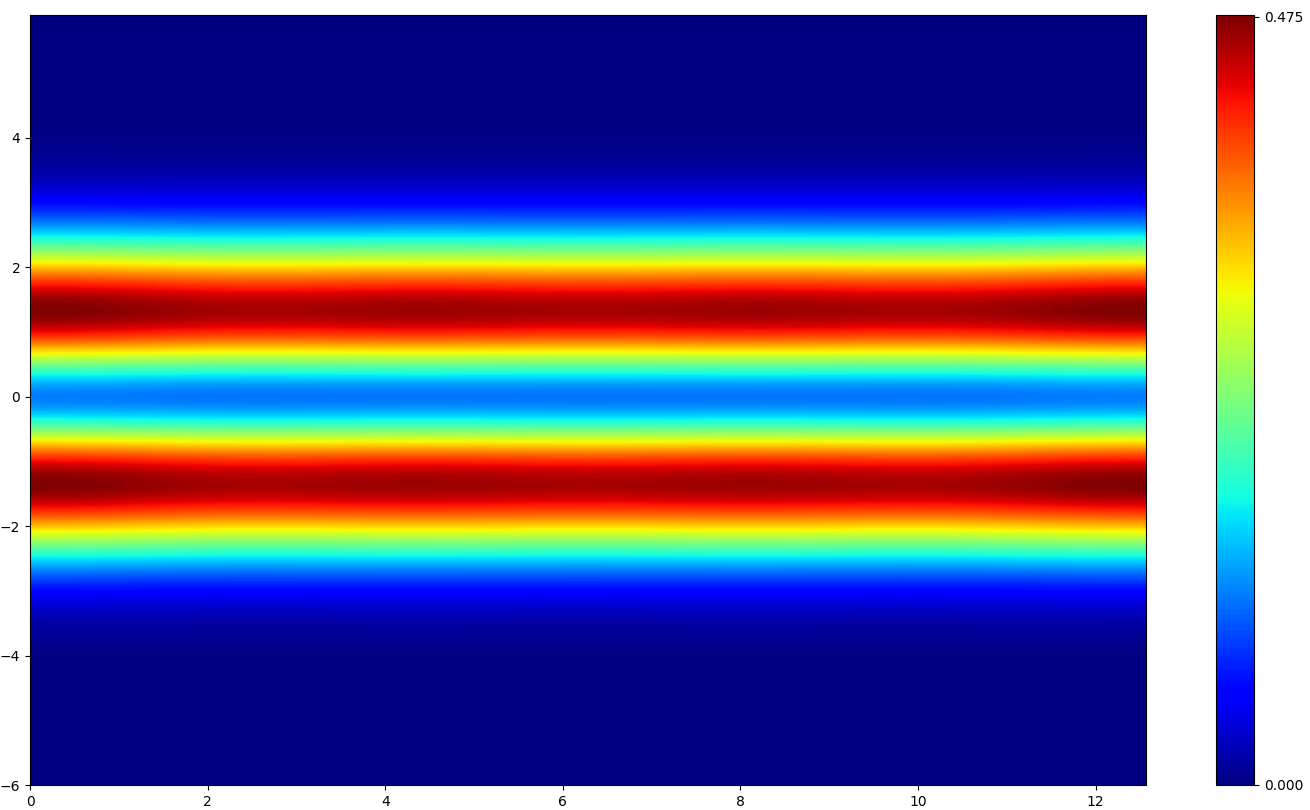}
  \caption{$t=0$}
\end{subfigure}
\hfill
\begin{subfigure}[b]{0.45\textwidth}
  \centering
  \includegraphics[width=\textwidth]{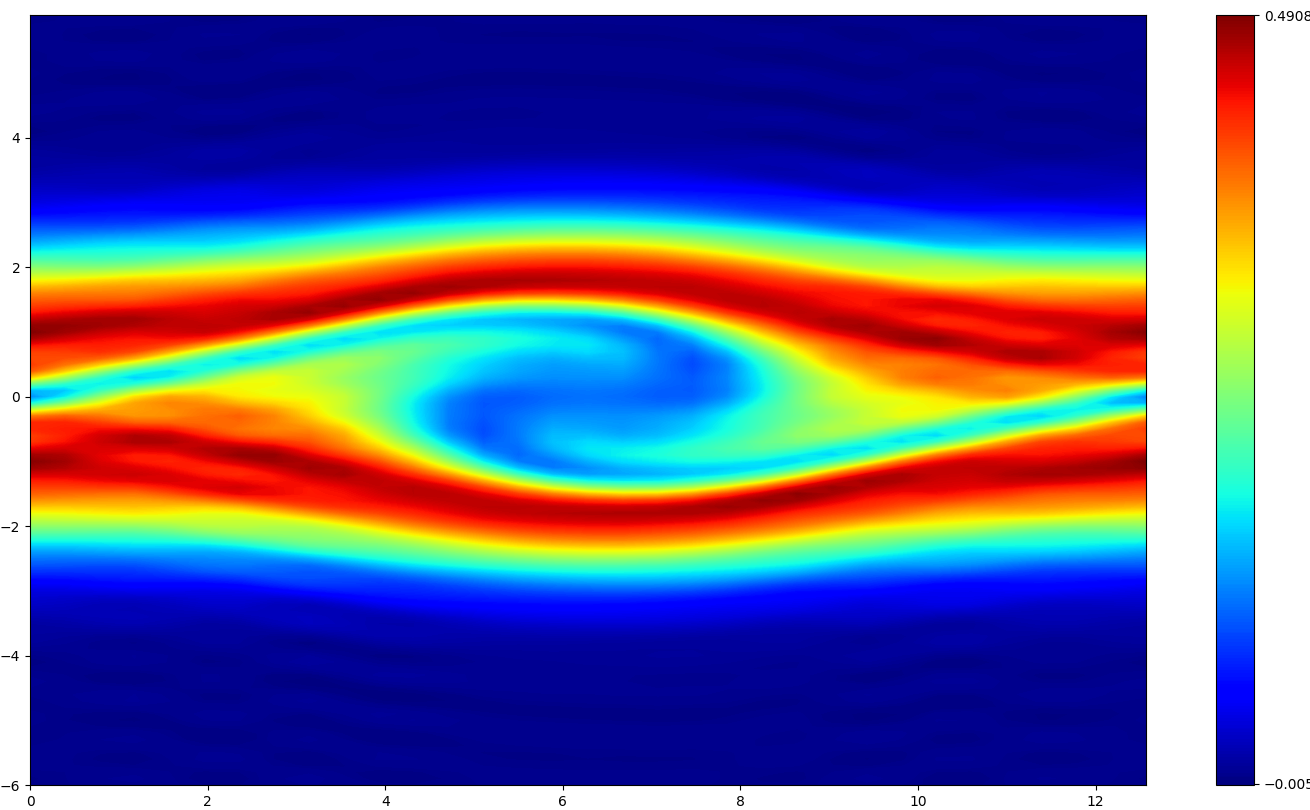}
  \caption{$t=20$}
\end{subfigure}
\caption{Density function for the two-stream instability computed with the Galerkin method at times $t=0$ and $t=20$.}
\label{fig:two_stream_galerkin}
\end{figure}
The results are shown in Figure \ref{fig:two_stream_pertrov_galerkin}, \ref{fig:two_stream_galerkin}.
The density function calculated at time $t=20$ is represented. 
The Galerkin method outperforms the Petrov-Galerkin method, 
as can be seen by comparing the distribution function plots 
in Figures~\ref{fig:two_stream_pertrov_galerkin} and~\ref{fig:two_stream_galerkin}. 
The numerical distribution function obtained with the Petrov-Galerkin method exhibits significantly stronger instability, 
while the Galerkin method maintains greater numerical stability.
In our opinion our numerical results  illustrate 
that the Galerkin method with AW Hermite functions 
has a potential for the computation of such non linear dynamics 
without any post-processing or filtering of the numerical results.

\subsection{Bump-on-tail instability}
We consider the problem of bump-on-tail instability, where the initial distribution function is characterized by a velocity profile that combines a Maxwellian distribution with a high-energy, warm beam. The initial distribution function is given by:
$$
f(0,x,v) = (g_p(v)+g_b(v)) (1+\kappa \cos(knx)),
$$
where $g_p(v)+g_b(v)$ represents the bump-on-tail distribution, defined as:
$$
\left\{
\begin{aligned}
g_p(v) &= \dfrac{n_p}{\sqrt{\pi} v_{th,p}} e^{-v^2/{v^2_{th,p}}}, \\
g_b(v) &= \dfrac{n_b}{\sqrt{\pi} v_{th,b}} e^{-(v-v_{d,b})^2/{v^2_{th,b}}} .
\end{aligned}
\right.
$$
In this expression, the primary ``plasma'' distribution \(g_p(v)\) is characterized by the number density $n_p$ and thermal velocity $v_{th,p}$. The ``bump'' distribution \(g_b(v)\), which introduces the instability, is described by the number density $n_b$, thermal velocity $v_{th,b}$, and drift velocity $v_{d,b}$.
The initail spatial function has cosinusoidal form, 
where $\kappa$ is the perturbation amplitude, 
$k$ is the mode number stimulated, and $L$ is the length.

The function $g_p(v)$, without drift, is exactly represented by the $0$th Hermite function 
with an appropriate $T$. For function $g_b(v)$, with drift, we consider 
the weighted Galerkin projection
\[
g_b(v) \approx \sum_{m=0}^N \langle g_b(v), \psi_m(v)\rangle_w \psi_m(v).
\]
With the definition of the function $\psi_m$, we have
\[
  \langle g_b(v), \psi_m(v)\rangle_w
= \dfrac{n_b}{\sqrt{\pi} v_{th,b}} \dfrac{1}{\sqrt{2^m m!}}
  \int_{\mathbb{R}} H_m(\dfrac{v}{\sqrt{T}}) e^{-\frac{(v-v_{d,b})^2}{v_{th,b}^2}} dv
\]
Denoting $\overline{v}= v/v_{th,b}$, we obtain
\begin{equation}\label{eq:linear_combination_bump_part}
  \langle g_b(v), \psi_m(v)\rangle_w
= \dfrac{n_b}{\sqrt{\pi} v_{th,b}} \dfrac{1}{\sqrt{2^m m!}}
  \int_{\mathbb{R}} H_m(\dfrac{v_{th,b}}{\sqrt{T}} \overline{v})
  e^{-\left(\overline{v}-\frac{v_{d,b}}{v_{th,b}}\right)^2} dv
\end{equation}
To evaluate the integral, we apply the formulas from \cite{gradshteyn2014table} for $\alpha, y \in \mathbb{R}$:
\begin{equation}\label{eq:gradshteyn_table}
\int_{\mathbb{R}} H_m(\alpha x) e^{-(x-y)^2} dx =
\left\{
\begin{aligned}
& \sqrt{\pi} y^m 2^m,
  \quad \alpha = 1, \\ 
& \sqrt{\pi} (1-\alpha^2)^{m/2} H_m \left( \dfrac{\alpha y}{\sqrt{1-\alpha^2}}\right),
  \quad \alpha \neq 1. \\ 
\end{aligned}
\right.
\end{equation}
Inserting \eqref{eq:gradshteyn_table} with \( \alpha = v_{th,b} / \sqrt{T}\), 
and $y = v_{d,b} / v_{th,b}$ into \eqref{eq:linear_combination_bump_part}, we obtain 
\begin{equation}
\langle g_b(v), \psi_m(v)\rangle_w =
\left\{
\begin{aligned}
& \dfrac{n_b}{\sqrt{T} \alpha} \dfrac{1}{\sqrt{2^m m!}}
  y^m 2^m,
  \quad \alpha = 1, \\ 
& \dfrac{n_b}{\sqrt{T} \alpha} \dfrac{1}{\sqrt{2^m m!}}
  (1-\alpha^2)^{m/2} H_m \left( \dfrac{\alpha y}{\sqrt{1-\alpha^2}}\right),
  \quad \alpha \neq 1. \\ 
\end{aligned}
\right.
\end{equation}

We consider a case with a strong perturbation, setting \( \kappa = 0.04 \), \( n = \sqrt{3} \), and \( k = 0.1 \). The parameters in $f_b(v)$ are chosen as \( n_p = 0.9 \), \( n_b = 0.1 \), \( v_{d,b} = 4.5 \), \( v_{th,p} = 2 \), and \( v_{th,b} = \sqrt{2}/2 \). These settings are consistent with those used in \cite{bessemoulin2022stability}.
Again we take $N=64,128,256$ and $J=16$ for both methods.

\begin{figure}[H]
    \centering
    \begin{subfigure}[b]{0.48\textwidth}
        \centering
        \includegraphics[width=\textwidth]{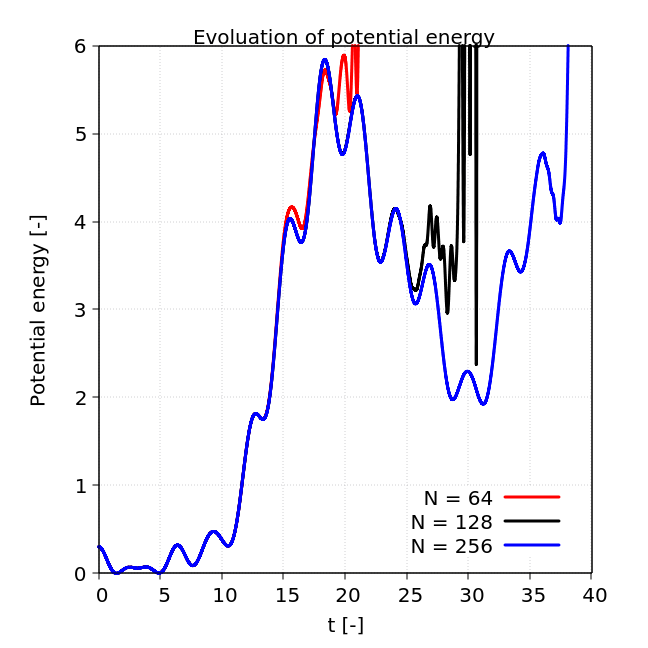}
    \end{subfigure}
    \begin{subfigure}[b]{0.48\textwidth}
        \centering
        \includegraphics[width=\textwidth]{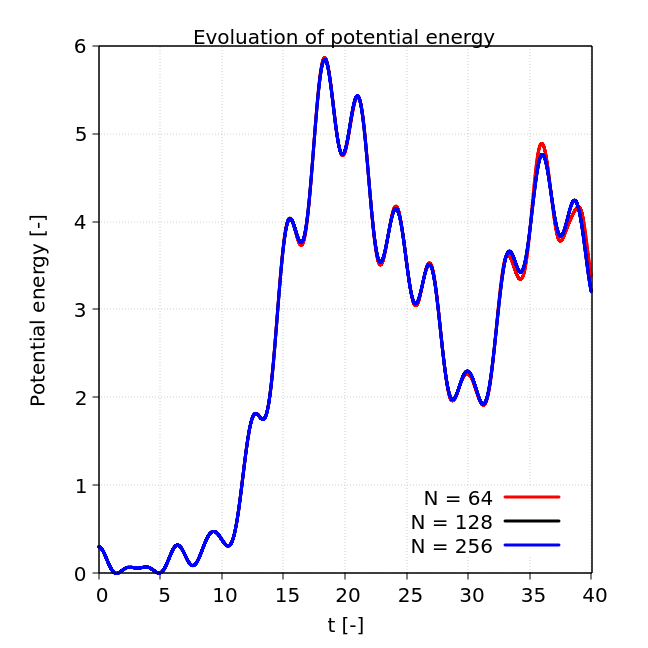}
    \end{subfigure}
    \caption{Bump-in-tail instability with strong perturbation: potential energy.}
    \label{fig:bump2:potential_energy}
\end{figure}
We plot the time evolution of the potential energy for both methods in Fig. \ref{fig:bump2:potential_energy}. 
The results from the Galerkin method exhibit a consistent structure and 
align well with those reported in \cite{bessemoulin2022stability}. 
In contrast, the results from the Petrov-Galerkin method exhibit instability and blow up over time.
\begin{figure}[H]
    \centering
    \begin{subfigure}[b]{0.48\textwidth}
        \centering
        \includegraphics[width=\textwidth]{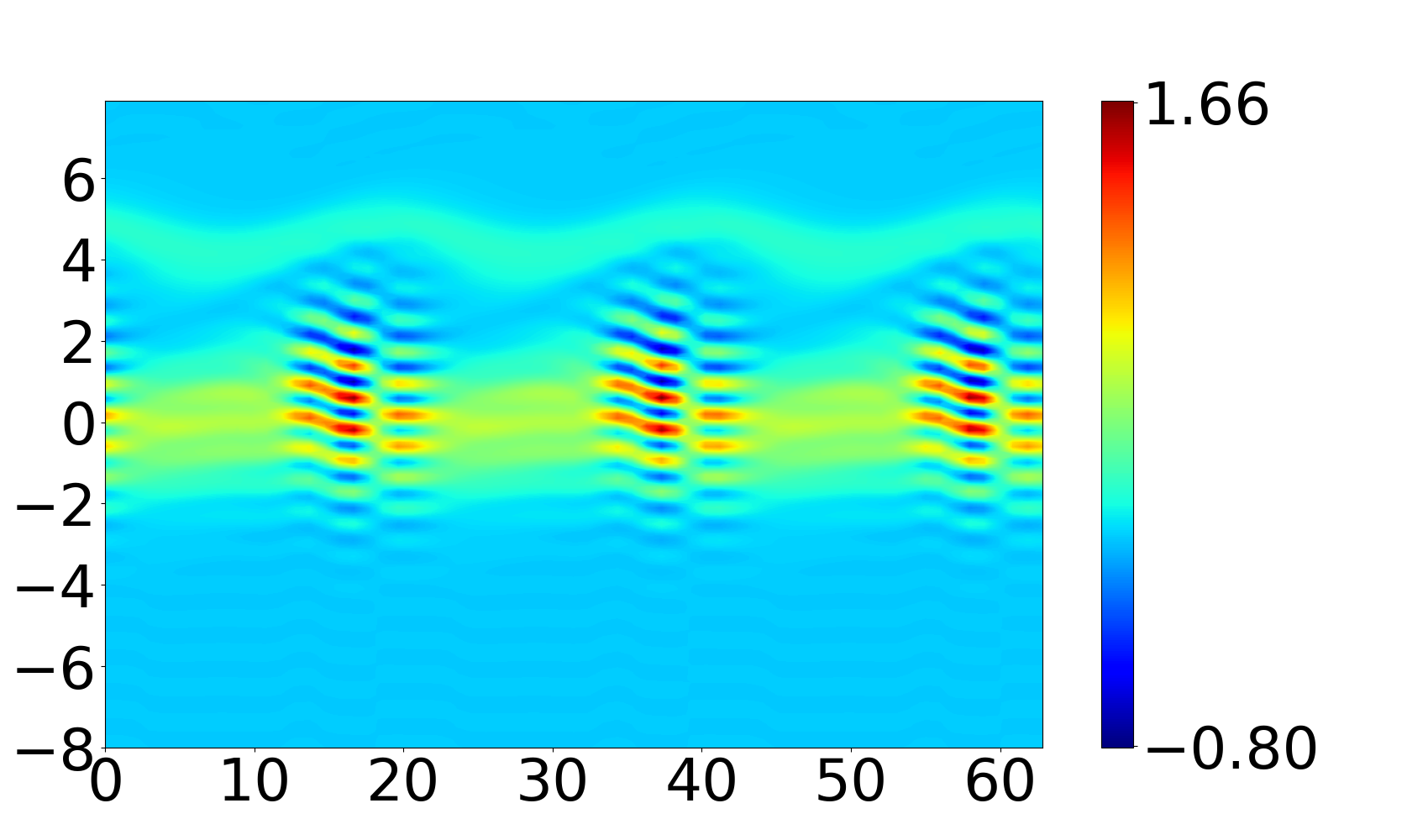}
	\caption{T = 10.0}
    \end{subfigure}
    \begin{subfigure}[b]{0.48\textwidth}
        \centering
        \includegraphics[width=\textwidth]{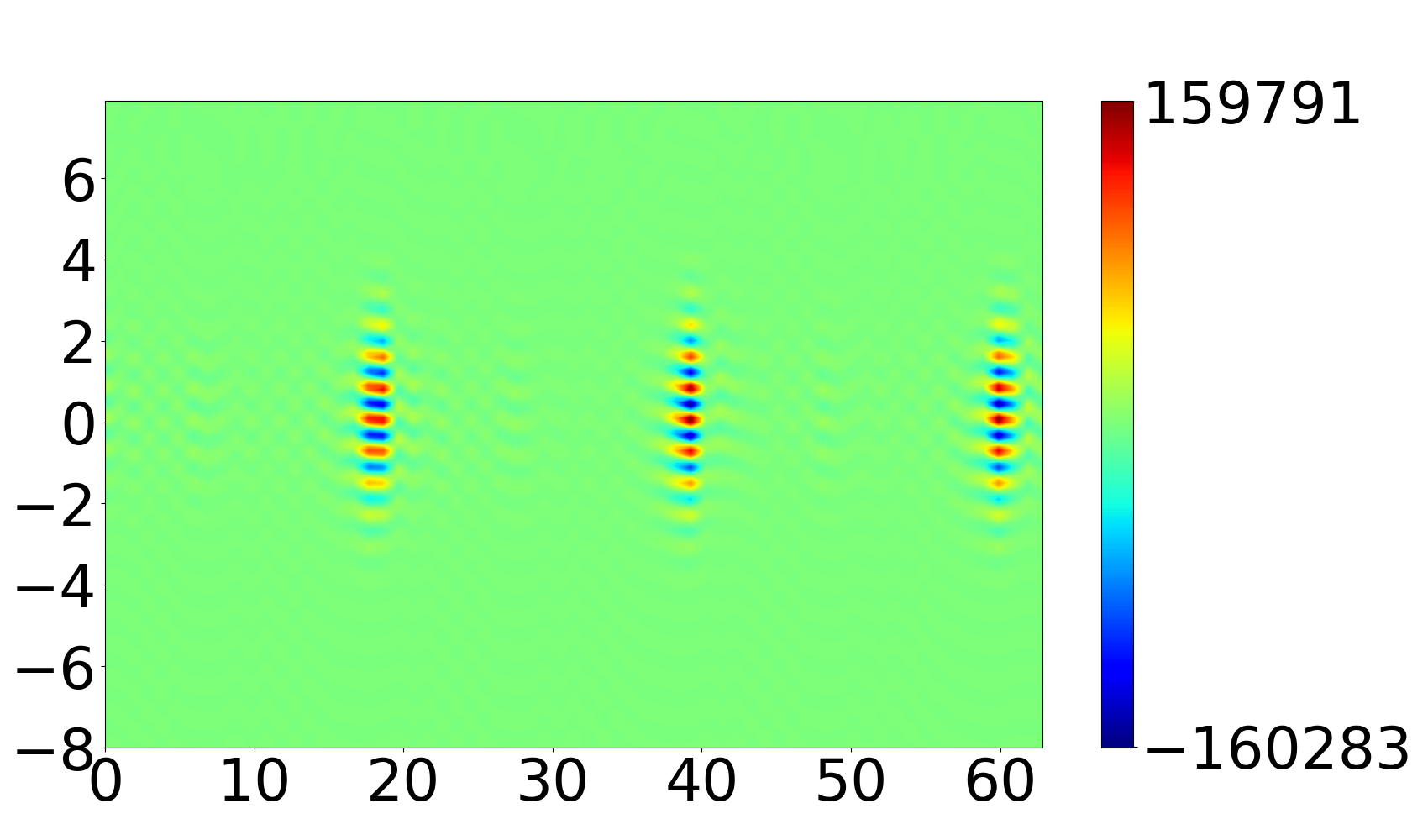}
	\caption{T = 20.0}
    \end{subfigure}
    \caption{Bump-in-tail instability with strong perturbation computed with the Petrov-Galerkin method: Plot of the distribution function $f$ at $t=10,20$, where $N=64$.}
    \label{fig:bump2:plot:Petrov-Galerkin}
\end{figure}
\begin{figure}[H]
    \centering
    \begin{subfigure}[b]{0.48\textwidth}
        \centering
        \includegraphics[width=\textwidth]{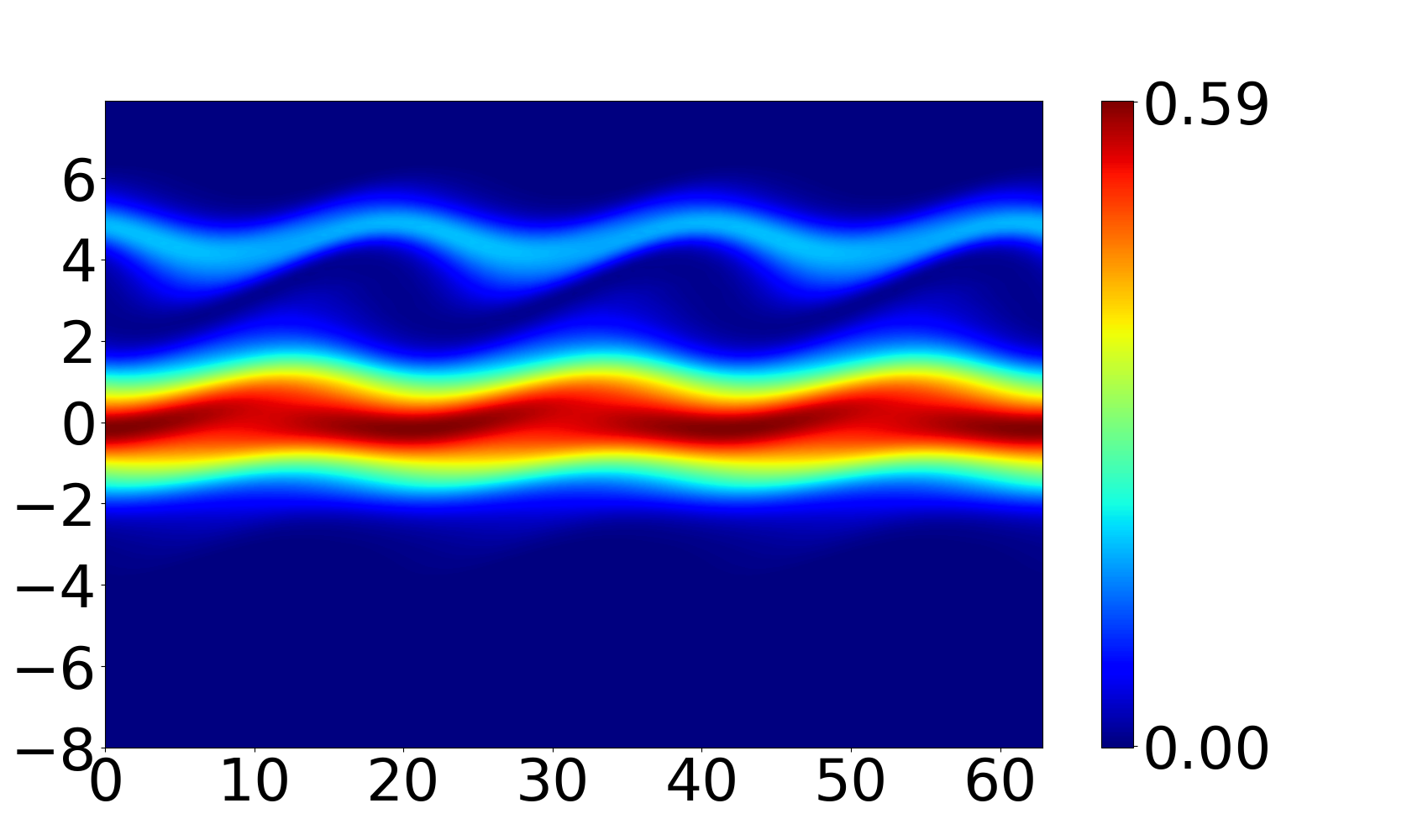}
	\caption{T = 10.0}
    \end{subfigure}
    \begin{subfigure}[b]{0.48\textwidth}
        \centering
        \includegraphics[width=\textwidth]{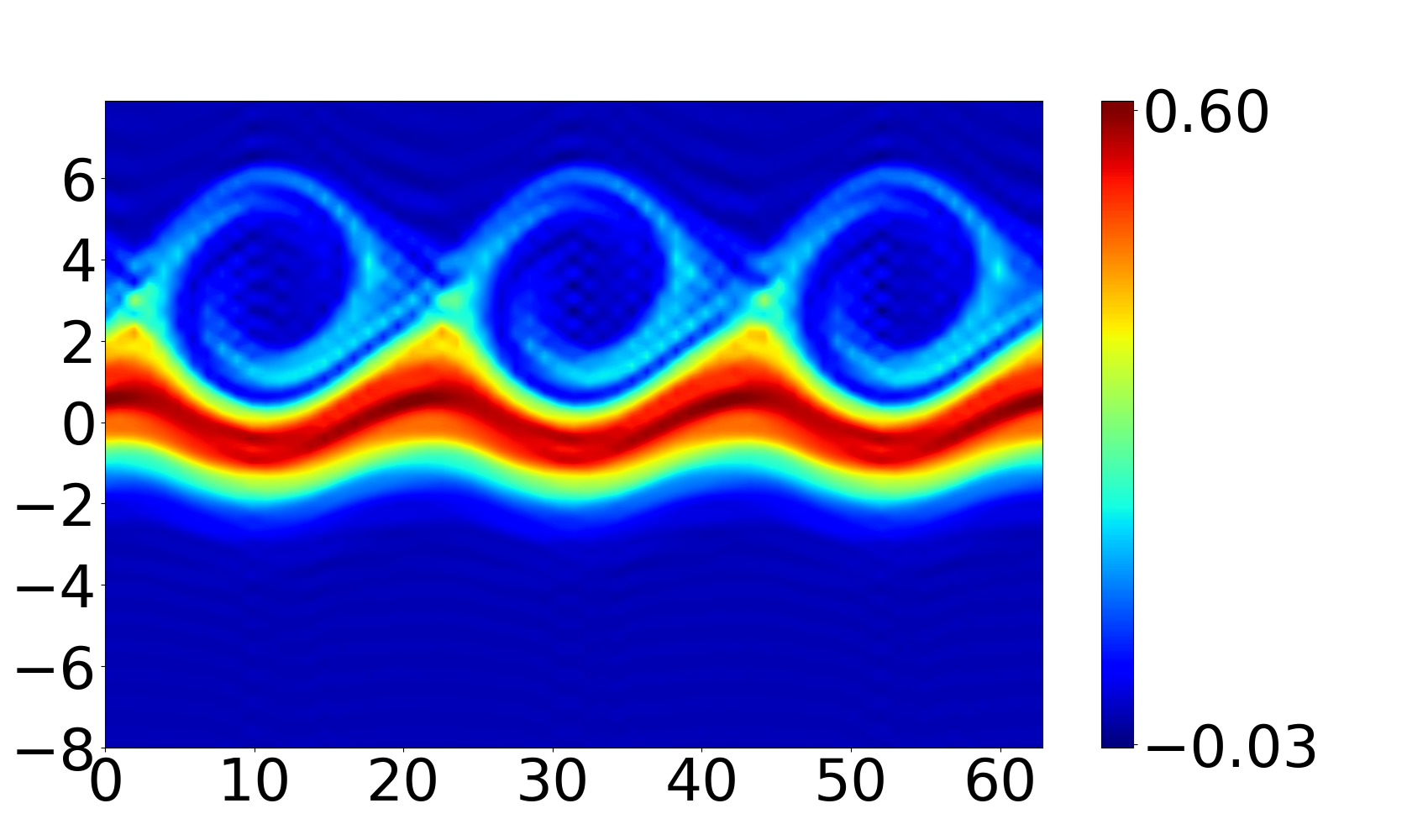}
	\caption{T = 20.0}
    \end{subfigure}
    \caption{Bump-in-tail instability with strong perturbation computed with the Galerkin method: Plot of the distribution function $f$ at $t=10,20$, where $N=64$.}
    \label{fig:bump2:plot:Galerkin}
\end{figure}
Finally, we present the surface plots of the distribution function at $t=10$ and $t=20$ 
in Fig. \ref{fig:bump2:plot:Petrov-Galerkin} \ref{fig:bump2:plot:Galerkin} for both methods. 
A comparison of the results reveals that, for the Petrov-Galerkin method, 
the solutions begin to exhibit instability at $t=10$ and completely blow up by $t=20$. 
In contrast, the Galerkin method produces relatively stable and well-behaved results over the same time intervals.

\section{Conclusion}
In this work, we explain why the numerical scheme based on the AW Hermite functions in the Petrov-Galerkin method applied to the VP system is unstable.
We propose in this paper a Galerkin method to the VP system with AW Hermite functions
allowing to stabilize the numerical solution. 
We also propose an equivalent form of which the resulting increase in computational cost remains modest.
The present work is the first stone to investigate the conservation properties and the convergence analysis of 
the proposed AW Hermite spectral discretization in the Galerkin method 
for the VP system. 

\appendix
\section{Analysis of the Gram matrix}\label{appendix:gram_matrix}

The coefficients of the matrix $A$ are 
$L(dxdv)^2$ scalar products of AW Hermite functions. 
These coefficients are computable in finite terms since 
the product of two AW Hermite functions 
can be expressed as a Gaussian function 
multiplied by a Hermite polynomial.
However, to our knowledge, 
the exact value of these coefficients is not available in the reference 
literature on special functions \cite{nist,szego1975orthogonal,magnus1967formulas}.
For demonstrating Proposition~\ref{prop:relation_interesting},
we calculate the quadratic scalar product of AW Hermite functions.
\begin{Theorem} \label{prop:4.1}
If the sum of the indices is odd $m+n\in 2\mathbb N+1$, 
then $a_{mn}=0$. Otherwise  
\begin{equation} \label{eq:b46}
  a_{mn}
= (-1)^{\frac{m-n}2} (\pi T)^{-\frac12} 2^{-(m+n) - \frac{1}2} 
  \frac{(m+n)!}{ \left( \frac{m+n}2 \right)! \sqrt{m!n!} }.
\end{equation}
\end{Theorem}
\begin{proof}
If $m+n$ is odd, then $\psi_m \psi_n$ is equal to a 
Gaussian function multiplied by an odd polynomial, 
so its integral vanishes. 
In this case $a_{mn}=0$. 
So let us consider the other case.

We have $\sqrt{ 2m/T } \  a_{mn}= \sqrt{ 2m/T } \ \int \psi_m(v) \psi_n(v)dv$. 
Using the general identity $(\psi_m)'(v)=-\sqrt{\frac{ 2(m+1)} T}\  \psi_{m+1}(v)$, 
we can write
$$
\sqrt{ 2m/T } \  a_{mn}= -  \int \psi_{m-1}'(v) \psi_n(v)dv=  \int \psi_{m}(v) \psi_n'(v)dv
$$
$$
=
- \sqrt{ 2(n+1)/T }   \int \psi_{m}(v) \psi_{n+1}(v)dv = - \sqrt{ 2(n+1)/T }   \  a_{m-1,n+1}.
$$
That is
\begin{equation} \label{eq:b40-rab}
\sqrt{ m } \  a_{mn}=- \sqrt{ n+1 }   \  a_{m-1,n+1}.
\end{equation}
We get by iteration
$$
{\left( m(m-1) \dots 2 \right) }^\frac12   a_{mn}= (-1)^m { \left( n(n+1) \dots (m+n)\right)}^\frac12 a_{0,m+n}
$$
that is
\begin{equation} \label{eq:b41}
a_{mn}=(-1)^m \left( {\frac{(n+m)!}{n!m!}} \right)^\frac12a_{0,m+n}.
\end{equation}
The technical Lemma \ref{lem:4.2} yields the value of $a_{0,m+n}$ from which we obtain 
$$
a_{mn}=(-1)^m \left( {\frac{(n+m)!}{n!m!}} \right)^\frac12(-1)^{(m+n)/2} (\pi T)^{-\frac12}2^{-(m+n) - \frac{1}2}  \frac{ (m+n)!^\frac12} { \left( \frac{m+n}2 \right)! }
$$
that is
$$
a_{mn}=(-1)^{\frac{m-n}2}(\pi T)^{-\frac12}2^{-(m+n) - \frac{1}2} \frac{(m+n)!}{ \left( \frac{m+n}2 \right)! \sqrt{m!n!} }.
$$
\end{proof}

\begin{lemma} \label{lem:4.2}
Let $m\in 2\mathbb N$. 
We have
$a_{0m}=(-1)^{m/2} (\pi T)^{-\frac12}2^{-m - \frac{1}2}  \frac{ (m!)^\frac12} { (m/2)! }$. 
\end{lemma}
\begin{proof}
We have
$\psi_0(v)\psi_m(v)=(\pi T)^{-1}\ (2^mm!)^{-\frac12} e^{-2v^2/T}H_m(v/\sqrt{T})$. 
To be able to  perform a rescaling in this expression, we can use  the general formula \cite[page 255]{magnus1967formulas}
$$
H_m(\lambda x)=  \sum_{l=0}^{[m/2]} \lambda^{m-2l} (\lambda^{2}-1)^l  \frac{m!}{(m-2l))!l!}H_{m-2l}(x).
$$ 
Take $\lambda=1/\sqrt 2$ and $x= \sqrt 2 v/\sqrt{T}$. Then
$$
H_m(v/\sqrt{T})= \left( -  \frac{1}{2} \right)^{m/2}
\frac{m!}{(m/2)!} +R(v)
$$
where the residual $R(v)$ is orthogonal to the weight $e^{-2v^2/T}$ 
because it is a linear combination of Hermite polynomials of degree $\geq 1$ (with convenient weight).
We obtain
$$
a_{0m}=\int \psi_0(v)\psi_m(v) dv =(\pi T)^{-1} (2^mm!)^{-\frac12}  \left( -  \frac{1}{2} \right)^{m/2}  \frac{m!}{(m/2)!}  \sqrt{T\pi/2}
$$
which yields the claim after simplification.
\end{proof}


\section{Proof of Proposition \ref{prop:relation_interesting}}\label{appendix:A}

This part is devoted to the proof of Proposition \ref{prop:relation_interesting}.
We define
$
\binom{b}{a} = 0, 
$
if $a$ or $b$ is not integer.
\begin{proof}
According Theorem~\ref{prop:4.1}, one has
\[
  a_{nm}
= (-1)^{\frac{n-m}{2}} (\pi T)^{-\frac{1}{2}} 2^{-(n+m)-\frac12}
  \dfrac{(n+m)!}{(\frac{n+m}{2})! \sqrt{n!m!}}, \quad\forall n+m \in 2\mathbb{N}.
\]
Using the value given in the claim, one has
\[
\begin{aligned}
  \sum_{m=0}^{N} a_{nm} z_{m,N+1}
=&(-1)^{\frac{n-(N+1)}{2}} (\pi T)^{-\frac{1}{2}} 2^{-(n+N+1)-\frac12} \\
 &\sum_{m=0}^{N} 
  (-1)^{\frac{N+1-m}{2}} 
  \dfrac{(n+m)!}{(\frac{n+m}{2})! n!m!} 
  \dfrac{n!(N+1)!}{\sqrt{n!(N+1)!}}
  \dfrac{1}{(\frac{N+1-m}{2})!} \\
=&(-1)^{\frac{n-(N+1)}{2}} (\pi T)^{-\frac{1}{2}} 2^{-(n+N+1)-\frac12} \\
 &S(N,n)
  \dfrac{n!(N+1)!}{(\frac{n+N+1}{2})! \sqrt{n!(N+1)!}},
\end{aligned}
\]
where 
\[
  S(N,n)
= \sum_{m=0}^{N} (-1)^{\frac{N+1-m}{2}}
  \dbinom{n+m}{n} \dbinom{\frac{n+N+1}{2}}{\frac{n+m}{2}}.
\]
Then using Lemma \ref{lemma:identity_S_N_n}, one has
\[
\begin{aligned}
  \sum_{m=0}^{N} a_{nm} z_{m,N+1}
=&(-1)^{\frac{n-(N+1)}{2}} (\pi T)^{-\frac{1}{2}} 2^{-(n+N+1)-\frac12} \\
 &S(N,n)
  \dfrac{n!(N+1)!}{(\frac{n+N+1}{2})! \sqrt{n!(N+1)!}} \\
=&-(-1)^{\frac{n-(N+1)}{2}} (\pi T)^{-\frac{1}{2}} 2^{-(n+N+1)-\frac12}
\dfrac{(n+N+1)!}{(\frac{n+N+1}{2})! \sqrt{n!(N+1)!}} \\
=&-a_{n,N+1}. 
\end{aligned}
\]
Moving the right-hand side to the left-hand side, we finish the proof.
\end{proof}
\begin{lemma}\label{lemma:identity_S_N_n}
$S(N,n) = -\dbinom{n+N+1}{n}, \quad\forall n+N+1 \in 2\mathbb{N}$.
\end{lemma}
\begin{proof}
We check the identity
\[
  S(N,n)
= \sum_{m=0}^{N+1} (-1)^{\frac{N+1-m}{2}}
  \dbinom{n+m}{n} \dbinom{\frac{n+N+1}{2}}{\frac{n+m}{2}}
- \dbinom{n+N+1}{n},
\]
We consider the case $N = 2M$ is even. 
Then $n$, $m$ are odd. We denote $n=2q+1$, $m=2k+1$, 
$S(N,n)$ can be rewritten as
\begin{equation}\label{eq:identity_to_check}
\begin{aligned}
  S(2M,n)
&=\sum_{k=0}^{M} (-1)^{M-k} 
  \dbinom{n+2k+1}{n} 
  \dbinom{\frac{n+2M+1}{2}}{\frac{n+2K+1}{2}}
- \dbinom{n+2M+1}{n}, \\
&=\sum_{k=0}^{M} (-1)^k P(2M, n, k) 
  \dbinom{M}{k}
- \dbinom{n+2M+1}{n}
\end{aligned}
\end{equation}
where $P(2M, n, k)$ is defined as
\[
\begin{aligned}
   P(2M, n, k)
&= (-1)^{M} \dfrac{((n+2M+1)/2)!}{n!M!}
   \dfrac{(n+2k+1)!k!}{(2k+1)!((n+2k+1)/2)!} \\
&= C(2M,n) 
   \dfrac{(2q+2k+2)!k!}{(2k+1)!(q+k+1)!}.
\end{aligned}
\]
By direct expansion, one checks $P(2M, n, k)$ can be written as 
a polynomial with respect to the variable $k$. 
To show this fact, define 
$A(q, k) = \dfrac{(2q+2k+2)!k!}{(2k+1)!(q+k+1)!}$.
It is clear that $A(0,k) = 2$. It is also clear that 
\[
  A(q+1, k)
= \dfrac{(2q+2k+4)(2q+2k+3)}{q+k+2} A(q,k)
= 2(2q+2k+3) A(q,k).
\]
By iteration, one has that $A(q,k)$ is a polynomial in $k$ of degree $q$.
So $P(2M, n, k)$ is also a polynomial in $k$ of degree $q=(n-1)/2 < M$.

On the other hand, one has the general identity for all degrees $r<M$ 
\[
\sum_{k=0}^{M} (-1)^k k^r \dbinom{M}{k} = 0, \quad r<M.
\]
Since $P(2M,n,k)$ is a polynomial in $k$ of the convenient degree, 
then the sum in \eqref{eq:identity_to_check} vanishes, which proves the case.

For the case $N = 2M-1$, $M\in\mathbb{N}^+$, 
the analysis is very similar to the case above.
\end{proof}

\end{document}